\documentclass[final
]{amsart}
\usepackage[text={6in,9in},centering]{geometry}
\usepackage[utf8]{inputenc}
\usepackage{hyperref}
\usepackage[english]{babel}
\usepackage{enumerate}
\usepackage{aliascnt}
\usepackage{amssymb}
\usepackage{amsmath}
\usepackage{amsthm}
\usepackage{verbatim}
\usepackage{mathtools}

\usepackage{esint} 

\usepackage{tikz}
\usepackage{mathscinet} 

\usepackage[title]{appendix}

\usepackage[notcite,notref]{showkeys}
\usepackage{color}

\numberwithin{equation}{section} 

\usepackage{dsfont} 
\usepackage{hyperref} 
\usepackage[noabbrev]{cleveref}

\usepackage{autonum} 

\allowdisplaybreaks[4] 

\newtheorem{prop}{Proposition}[section]

\newaliascnt{lem}{prop} 

\newtheorem{lem}[lem]{Lemma}

\aliascntresetthe{lem} 
\Crefname{lem}{Lemma}{Lemmas}

\newaliascnt{defi}{prop} 
 \newtheorem{defi}[defi]{Definition}

\aliascntresetthe{defi} 
\Crefname{defi}{Definition}{Definitions}

\newaliascnt{cor}{prop} 
 \newtheorem{cor}[cor]{Corollary}
 
\aliascntresetthe{cor}

\newaliascnt{remark}{prop} 
 \newtheorem{remark}[remark]{Remark}
 
\aliascntresetthe{remark} 

\newaliascnt{thm}{prop} 
 \newtheorem{thm}[thm]{Theorem}
 
\aliascntresetthe{thm} 

\newaliascnt{example}{prop} 
 
\aliascntresetthe{example}

\def\equationautorefname~#1\null{%
  (#1)\null
}

\DeclareMathOperator{\diam}{diam} 
\newcommand{\R}{\ensuremath{\mathbb{R}}}
\newcommand{\N}{\ensuremath{\mathbb{N}}}

\renewcommand{\S}{\ensuremath{\mathbb{S}}}

\newcommand*\diff{\mathop{}\!\mathrm{d}}

\newcommand{\defeq}{\vcentcolon=}
\newcommand{\eqdef}{=\vcentcolon}

\newcommand{\CalA}{\ensuremath{\mathcal{A}}}
\newcommand{\CalW}{\ensuremath{\mathcal{W}}}
\newcommand{\CalI}{\ensuremath{\mathcal{I}}}

\newcommand{\abs}[1]{\ensuremath{\lvert #1\rvert}}
\newcommand{\Abs}[1]{\ensuremath{\left\lvert #1\right\rvert}}
\newcommand{\Norm}[2]{\ensuremath{\left\Vert #1 \right\Vert_{#2}}}
\newcommand{\norm}[2]{\ensuremath{\Vert #1 \Vert_{#2}}}

\newcommand{\dtzero}{\left.\frac{\diff}{\diff t}\right\vert_{t=0}}

\DeclareMathOperator{\CalV}{\ensuremath{\mathcal{V}}}

\usepackage{color}
\newcommand{\FFF}{\color{black}}
\newcommand{\EEE}{\color{black}}

\makeatletter
\@namedef{subjclassname@2020}{$2020$ Mathematics Subject Classification}
\makeatother

\begin{document}
\title{The Willmore flow with prescribed isoperimetric ratio}	
\author[F.~Rupp]{Fabian Rupp}
\address[F.~Rupp]{Faculty of Mathematics, University of Vienna, Oskar-Morgenstern-Platz 1, 1090 Vienna, Austria.}
\email{fabian.rupp@univie.ac.at}


\maketitle
\begin{abstract}
We introduce a non-local $L^2$-gradient flow for the Willmore energy of immersed surfaces which preserves the isoperimetric ratio. For spherical initial data with energy below an explicit threshold, we show long-time existence and convergence to a Helfrich immersion. This is in sharp contrast to the locally constrained flow, where finite time singularities occur. 
\end{abstract}

\section{Introduction and main results}

Finding the shape which encloses the maximal volume among surfaces of prescribed area is certainly one of the oldest and yet most prominent problems in mathematics and goes back to the legend of the foundation of Carthage. Since then generations of mathematicians have been studying \emph{isoperimetric problems}, \FFF aiming to find \EEE the best possible shape in all kinds of \FFF settings. \EEE It turns out that --- by the \emph{isoperimetric inequality} --- \FFF the optimal configuration in Euclidean space is given by a round sphere. \EEE 

Likewise, the round spheres are the absolute minimizers for the \emph{Willmore energy}, a functional measuring the bending of an immersed surface with various applications also beyond geometry, for instance in the study of biological membranes \cite{Canham,Helfrich}, general relativity \cite{Hawking}, nonlinear elasticity \cite{FJM02} and image restoration \cite{DroskeRumpf}.

Note that the round spheres describe the optimal shape in both situations. 
In this article, we will study their relation using a gradient flow approach.

For an immersion $f\colon\Sigma\to\R^3$ of a closed oriented surface $\Sigma$, its Willmore energy is defined by
\begin{align}\label{eq:defWillmore}
	{\CalW}(f) \defeq \frac{1}{4}  \int_{\Sigma} \abs{H}^2\diff \mu.
\end{align}
Here $\mu=\mu_f$ denotes the area measure induced by the pull-back of the Euclidean metric $g_f \defeq f^*\langle \cdot, \cdot\rangle$, and $H  = H_f\defeq \langle \vec{H}_f, \nu_f\rangle$ denotes the (scalar) mean curvature with respect to $\nu=\nu_f \colon\Sigma\to \S^2$, the unique unit normal along $f$ induced by the chosen orientation on $\Sigma$, see \eqref{eq:def normal} below.
A related quantity is the \emph{umbilic Willmore energy}, given by 
\begin{align}
	\CalW_0(f) \defeq \int_{\Sigma}\abs{A^0}^2\diff\mu,
\end{align}
where $A^0$ denotes the trace-free part of the second fundamental form. As a consequence of the Gauss--Bonnet theorem, these two energies are equivalent from a variational point of view, since for a surface with fixed genus $g$, we have
\begin{align}\label{eq:WvstildeW}
	\CalW_0(f) =  2{\CalW}(f)-8\pi + 8\pi g.
\end{align}
Both energies are not only \emph{geometric}, i.e.\ invariant under diffeomorphisms on $\Sigma$, but --- remarkably --- also \emph{conformally invariant}, i.e.\ invariant with respect to smooth Möbius transformations of $\R^3$. By \FFF \cite[Theorem 7.2.2]{MR1261641}\EEE, we have $\CalW(f)\geq 4\pi$ with equality if and only if $\Sigma=\S^2$ and $f\colon \S^2\to\R^3$ parametrizes a round sphere. 

\smallskip
The \emph{isoperimetric ratio} of an immersion $f\colon\Sigma\to\R^3$ is defined as the quotient
\begin{align}
	\CalI(f) &\defeq 36\pi\frac{\CalV(f)^2}{\CalA(f)^3},\quad\text{where } \label{eq:defI}\\
	\CalA(f) &\defeq \int_{\Sigma}\diff \mu\quad\text{and}\quad \CalV(f)\defeq  -\frac{1}{3} \int_\Sigma \langle f, \nu\rangle\diff \mu
\end{align}
denote the area and the \FFF algebraic \EEE volume enclosed by $f(\Sigma)$, respectively. Here,
the normalizing constant is chosen such that 
by the isoperimetric inequality we always have $\CalI(f)\eqdef\sigma\in [0,1]$ with $\sigma=1$ if and only if $\Sigma=\S^2$ and $f\colon\S^2\to\R^3$ parametrizes a round sphere. 
Critical points of the isoperimetric ratio --- or equivalently, critical points of the volume functional with prescribed area --- are precisely the \emph{CMC-surfaces}, i.e.\ the surfaces with constant mean curvature, which form an important generalization of minimal surfaces and naturally arise in the modeling of soap bubbles. 

The problem of minimizing the Willmore energy among all immersions of a genus $g$ surface $\Sigma_g$ with prescribed isoperimetric ratio, i.e.\ the minimization problem \begin{align}\label{eq:defbeta}
	\beta_g(\sigma) \defeq \inf\left\lbrace \CalW(f) \mid f\colon \Sigma_g \to \R^3 \text{ immersion with } \CalI(f)=\sigma\right\rbrace,
\end{align}
naturally arises in mathematical biology in the \emph{Canham--Helfrich model} \cite{Canham,Helfrich} with zero spontaneous curvature and
has already been studied mathematically in  \cite{Schygulla,KellerMondinoRiviere,MondinoScharrer2}.
While the genus zero case was solved in \cite{Schygulla}, the results in \cite{KellerMondinoRiviere,MondinoScharrer2} combined with recent findings in \cite{Scharrer} and \cite{KusnerMcGrath} show that the infimum in \eqref{eq:defbeta} is always attained for any $g\in \N_0$ and $\sigma\in (0,1)$; and satisfies $\beta_g(\sigma)<8\pi$.
The energy threshold $8\pi$ also plays an important role in the analysis of the Willmore energy, since by the famous Li--Yau inequality \cite{LiYau}, any immersion $f$ of a compact surface with $\CalW(f)<8\pi$ has to be embedded. 

A sufficiently smooth minimizer in \eqref{eq:defbeta} is a \emph{Helfrich immersion}, i.e.\ a solution to the Euler--Lagrange equation
\begin{align}\label{eq:Helfrich eq}
	\Delta H + \abs{A^0}^2H - \lambda_1 H -\lambda_2 = 0\quad \text{ for some }\lambda_1,\lambda_2\in \R,
\end{align}
where $\Delta=\Delta_{g_f}$ denotes the Laplace--Beltrami operator on $(\Sigma, g_f)$.
In \cite{McCoyWheelerClassification}, solutions  to \eqref{eq:Helfrich eq} with small umbilic Willmore energy have been classified, depending on the sign of the \emph{Lagrange-multipliers} $\lambda_1$ and $\lambda_2$. 
We observe that for $\lambda_1, \lambda_2\in \R$ fixed, \eqref{eq:Helfrich eq} is also the Euler--Lagrange equation of the \emph{Helfrich energy} given by
\begin{align}\label{eq:defHelrich}
	\mathcal{H}_{\lambda_1, \lambda_2}(f)\defeq \CalW_0(f)+\lambda_1\CalA(f)+\lambda_2\CalV(f),
\end{align}
where the energy either penalizes or favors large area or volume, depending on the sign of $\lambda_1$ and $\lambda_2$, respectively.
\smallskip 

The $L^2$-gradient flow of the Willmore energy was introduced and studied by Kuwert and Schätzle in their seminal works \cite{KSGF,KSSI,KSRemovability}.
\FFF 
Their methods are very robust and allow to handle also other situations, such as the  \emph{surface diffusion flow} \cite{McCoyWheelerWilliams,MR2898774} and the Willmore flow of \emph{tori of revolution} \cite{DMSS20}. \EEE
The locally constrained \emph{Helfrich flow}, i.e.\ the $L^2$-gradient flow for the energy \eqref{eq:defHelrich}, and its asymptotic behavior have been studied in \cite{McCoyWheeler,BlattHelfrich}, where it was shown that finite time singularities must occur below a certain energy threshold. However, this flow does not preserve the isoperimetric ratio.

The goal of this article is to discuss a dynamic version of the minimization problem \eqref{eq:defbeta}. To this end, we introduce the \emph{Willmore flow with prescribed isoperimetric ratio}, which decreases $\CalW$ as fast as possible while keeping $\CalI(f)\equiv\CalI(f_0)=\sigma$ fixed. This yields the evolution equation
\begin{align}\label{eq:IsoWF}
	\partial_t f &= \left[-\Delta H - \abs{A^{0}}^2H +  \lambda
	\left(\frac{3}{\CalA(f)} H - \frac{2}{\CalV(f)}\right)\right]\nu,
\end{align}
where the \emph{Lagrange multiplier} $\lambda \defeq \lambda(t)\defeq \lambda(f_t)$ depends on $f_t \defeq f(t,\cdot)$ and is given by
\begin{align}\label{eq:deflambda}
	\lambda(f) \defeq \frac{\int \Big(\Delta H +\abs{A^{0}}^2H\Big)\Big(\frac{3}{\CalA(f)} H - \frac{2}{\CalV(f)}\Big)\diff \mu}{\int\abs{ \frac{3}{\CalA(f)} H - \frac{2}{\CalV(f)}}^2\diff \mu}.
\end{align}
	In \eqref{eq:dtI=0} below we will justify the particular choice of $\lambda$, which yields that $\CalI$ is actually preserved along a solution of \eqref{eq:IsoWF}--\eqref{eq:deflambda}.
\begin{defi}\label{def:IsoWF}
	Let $\sigma\in (0,1), T>0$ and let $\Sigma_g$ denote a connected, oriented and closed surface with genus $g\in \N_0$. A smooth family of immersions $f\colon[0,T)\times \Sigma_g\to\R^3$ satisfying \eqref{eq:IsoWF} with $\lambda$ as in \eqref{eq:deflambda} and  $\CalI(f)\equiv \sigma$ is called a \emph{$\sigma$-isoperimetric Willmore flow} with \emph{initial datum} $f_0\defeq f(0,\cdot)$.
\end{defi}
Stationary solutions of the flow \eqref{eq:IsoWF}--\eqref{eq:deflambda} are solutions to the Helfrich equation \eqref{eq:Helfrich eq} for $\lambda_1 = \frac{3}{\CalA(f)}\lambda$ and $ \lambda_2 = - \frac{2}{\CalV(f)}\lambda$. Conversely, any Helfrich immersion is also a stationary solution to \eqref{eq:IsoWF}--\eqref{eq:deflambda}, see \Cref{lem:Helfrich vs stationary} below.

However, as the Lagrange multiplier $\lambda$ defined in \eqref{eq:deflambda} depends on the solution, the isoperimetric flow \eqref{eq:IsoWF} substantially differs from the $L^2$-gradient flow of the Helfrich energy \eqref{eq:defHelrich}, where the parameters $\lambda_1$ and $\lambda_2$ are fixed numbers and chosen \emph{a priori}. On the analytic side, the integral nature of the Lagrange multiplier makes the evolution equation \eqref{eq:IsoWF} a \emph{non-local}, {quasilinear}, {degenerate parabolic} PDE of 4\textsuperscript{th} order. Also geometrically, the constraint $\CalI(f)\equiv\sigma$ causes new difficulties, as we cannot control the area and the volume independently along the flow (as in \cite{BlattHelfrich}, for instance), but only the isoperimetric ratio $\CalI$. 

The Willmore flow with a constraint on either the area or the enclosed volume has been studied in \cite{Jachan} and a recent article by the author \cite{RuppVolumePreserving}. However, the situation here is fundamentally different and several new challenges arise. 

First, if only the area or the volume is prescribed (and nonzero), constrained critical points of the corresponding variational problem are in fact \emph{Willmore immersions,} i.e.\ solutions of \eqref{eq:Helfrich eq} with $\lambda_1=\lambda_2=0$, due to the scaling invariance of the Willmore energy. Although still an active field of research, the classification of these Willmore immersions is much better understood than that of general solutions of (1.4) and a crucial ingredient in classifying the blow-ups in \cite{RuppVolumePreserving}.
Second, in \cite{RuppVolumePreserving} the different scaling of the energy and constraint has been used to represent the Lagrange multiplier in a way that allows for good \emph{a priori estimates.} This neat trick is clearly not available for the flow \eqref{eq:IsoWF}--\eqref{eq:deflambda}.
Third, unlike in \cite{RuppVolumePreserving}, the Lagrange multiplier has a much more complicated algebraic structure and cannot be treated as a lower order term.

These obstructions are the reason for a new energy threshold in the following main result on global existence and convergence.

%

\begin{thm}\label{thm:convergence main}
Let $f_0\colon \S^2\to\R^3$ be a smooth immersion with $\CalI(f_0)=\sigma\in (0,1)$ and such that $\CalW(f_0)\leq \min\left\{\frac{4\pi}{\sigma},8\pi\right\}$. Then there exists a unique $\sigma$-isoperimetric Willmore flow with initial datum $f_0$. This flow exists for all times and, as $t\to\infty$, it converges smoothly after reparametrization to a Helfrich immersion $f_{\infty}$ with $\CalI(f_\infty)=\sigma$ solving \eqref{eq:Helfrich eq} with $\lambda_1\neq 0$ and $\lambda_2 \neq 0$.
\end{thm}
This shows a fundamentally different behavior of the isoperimetric Willmore flow and the Helfrich flow, where finite time singularities occur, cf.\ \cite{BlattHelfrich,McCoyWheeler}. Consequently, despite its new analytic challenges, the introduction of the non-local Lagrange multiplier has a \emph{regularizing} effect on the gradient flow, see also \cite{Huisken} for a related result for the mean curvature flow.

\FFF The $\frac{4\pi}{\sigma}$-threshold in \Cref{thm:convergence main} is motivated by the following simple application of the triangle inequality in $L^2(\diff\mu)$. With $\mathcal{I}(f)=\sigma$ and \eqref{eq:defI}, 
we have
\begin{align}\label{eq:triangle estimate}
		\int_{\Sigma}\Abs{\frac{3}{\CalA(f)} H - \frac{2}{\CalV(f)}}^2\diff \mu \geq \frac{36}{\CalA(f)^2}\left(\sqrt{\frac{4\pi}{\sigma}}- \sqrt{\CalW(f_0)}\right)^2.
\end{align}
This estimate bounds the denominator in \eqref{eq:deflambda} from below if $\mathcal{W}(f_0)<\frac{4\pi}{\sigma}$. Moreover, it allows to control the Lagrange multiplier in the crucial estimates by essentially lower order quantities, see \Cref{sec:lambda}.  \EEE

We highlight that the assumption in \Cref{thm:convergence main} is not an implicit smallness of the initial energy, cf.\ \cite{KSSI,MR2898774}, but the threshold is explicitly given, although very little is known about minimizers and critical points of \eqref{eq:defbeta}. Moreover, as $\sigma\nearrow 1$, the interval of admissible initial energies in \Cref{thm:convergence main} becomes arbitrarily small. This seems plausible, since if $\sigma=1$, $f_0$ is a round sphere and the denominator in \eqref{eq:deflambda} vanishes. Thus, it is a priori unclear whether there exists an admissible immersion $f_0$ in \Cref{thm:convergence main} if $\sigma\in (\frac12, 1)$ --- in fact, this is equivalent to the condition $\beta_0(\sigma)\leq \frac{4\pi}{\sigma}$. In \Cref{thm:beta0<4pi/sigma} below, we will prove $\beta_0(\sigma)< \frac{4\pi}{\sigma}$ for $\sigma\in (0,1)$, which is asymptotically sharp as $\sigma\nearrow 1$, and consequently the existence of a suitable $f_0$ follows.
%
We also point out that it is unknown if the energy threshold in \Cref{thm:convergence main} is optimal as it is for the classical Willmore flow \cite{Blatt,DMSS20}.

The proof of \Cref{thm:convergence main} is based on the methods developed by Kuwert--Sch\"atzle for the Willmore flow \cite{KSGF,KSSI,KSRemovability}. Under a non-concentration assumption on the curvature, we use localized energy estimates to control the evolution, see \Cref{sec:energy estimates} below. However, as in \cite{RuppVolumePreserving}, these estimates depend on certain $L^p$-type bounds on $\lambda$. The key ingredient of this paper is that for locally small curvature and if the initial energy is below the threshold of \Cref{thm:convergence main}, the Lagrange multiplier can be absorbed in the estimates, see \Cref{sec:lambda}, in particular \Cref{lem: lambda nenner bound,lem:lambda new}. This is an essential observation, which we can use to prove a lower bound on the lifespan and to construct a blow-up limit in the spirit of \cite{KSSI}, see \Cref{sec:blowup}. Using the control over the Lagrange multiplier in the energy regime of \Cref{thm:convergence main}, we deduce a crucial \emph{rigidity result:} either the blowup is a \emph{compact} Helfrich immersion or a \emph{Willmore immersion,} see \Cref{lem:blowup existence}. In the first case, we conclude global existence and convergence by an argument based on the {\L}ojasiewicz--Simon inequality in the spirit of \cite{CFS09}, combined with recent progress on this inequality in the presence of constraints \cite{Rupp}. Due to the rigidity of the blow-up, we can follow the inversion strategy in \cite{KSRemovability} relying on the classification of compact Willmore spheres \cite{Bryant1984} to exclude the second case.

\FFF 
This last step is also where we crucially make use of the assumption $\Sigma_g=\S^2$. 
In the case of higher genus, a classification result for Willmore surfaces as in [5] is currently lacking. Even if such a classification were available, a precise comprehension of the behavior under inversion would be indispensable to extend the argument beyond the spherical case.
However, since \EEE the blow-up analysis is also available if $g\geq 1$, we establish the following remarkable dichotomy result.	
\begin{cor}\label{cor}
	Let $\sigma\in (0,1)$, let $\Sigma$ be a closed, oriented and connected surface and suppose that $f\colon [0,T)\times \Sigma\to\R^3$ is a maximal $\sigma$-isoperimetric Willmore flow such that $\CalW(f_0)<\frac{4\pi}{\sigma}$. Then there exist $\hat{c}\in (0,1)$, $(t_j)_{j\in \N}\subset [0,T), t_j\nearrow T, (r_j)_{\in \N}\subset (0,\infty)$ and $(x_j)_{j\in \N}\subset \R^3$ such that the sequence of immersions
	\begin{align}
		\hat{f}_j\defeq r_j^{-1} \left(f(t_j+r_j^4\hat{c}, \cdot)-x_j\right)
	\end{align}
	converges, as $j\to\infty$, smoothly on compact subsets of $\R^3$ after reparametrization to a proper Helfrich immersion $\hat{f}\colon\hat{\Sigma}\to\R^3$ where $\hat{\Sigma}\neq\emptyset$ is a complete surface without boundary. Moreover
	\begin{enumerate}[(a)]
		\item if $\hat{\Sigma}$ is compact, then $T=\infty$ and, as $t\to\infty$, the flow $f$ converges smoothly after reparametrization to a Helfrich immersion $f_\infty$ as $t\to\infty$.
		\item if $\hat{\Sigma}$ is not compact, then $\hat{f}$ is a Willmore immersion.
	\end{enumerate}
\end{cor}
Hence, under the above assumptions, in the singular case (b) the influence of the (non-local) constraint vanishes after rescaling as $t\to\infty$ and the purely local term in \eqref{eq:IsoWF}, coming from the Willmore functional, dominates.

We now outline the structure of this article. After a brief review of the most relevant analytic and geometric background in \Cref{sec:prelm}, we start our analysis by carefully computing and estimating a localized version of the energy decay in \Cref{sec:energy estimates}. In \Cref{sec:lambda}, we control the Lagrange multiplier in the energy regime of \Cref{thm:convergence main} which then enables us to construct a blow-up limit in \Cref{sec:blowup}. Finally, in \Cref{sec:convergence} we prove our convergence result, \Cref{thm:convergence main}, and \Cref{cor} before we show \Cref{thm:beta0<4pi/sigma} in \Cref{sec:beta}, yielding that the set of admissible initial data in \Cref{thm:convergence main}  is always non-empty.

\section{Preliminaries}\label{sec:prelm}

In this section, we will briefly review the geometric and analytic background and prove some first properties of the flow \eqref{eq:IsoWF}, see also \cite{KSLectureNotes} for a more detailed discussion.

\subsection{Geometric and analytic background}

In the following, $\Sigma_g$ always denotes an abstract compact, connected and oriented surface of genus $g\in \N_0$ without boundary.

An immersion $f\colon\Sigma_g\to\R^3$  induces the pullback metric $g_f=f^{\ast}\langle\cdot, \cdot\rangle$ on $\Sigma_g$, which in local coordinates is given by
\begin{align}
	g_{ij}\defeq \langle\partial_i f, \partial_j f\rangle,
\end{align} 
where $\langle \cdot, \cdot\rangle$ denotes the Euclidean metric. The chosen orientation on $\Sigma_g$ determines a unique smooth unit normal field $\nu\colon\Sigma_g\to\S^2$ along $f$, which in local coordinates in the orientation  is given by
\begin{align}\label{eq:def normal}
	\nu = \frac{\partial_1 f\times \partial_2 f}{\abs{\partial_1 f\times \partial_2 f}}.
\end{align} 
We will always work with this unit normal vector field.

The (scalar) second fundamental form of $f$ is then given by $ A_{ij} \defeq \langle\partial_i\partial_j f, \nu\rangle$ and the mean curvature and the tracefree part of the second fundamental form are defined as
\begin{align}
	H \defeq g^{ij}A_{ij} \text{ and }  A^{0}_{ij} \defeq  A_{ij} - \frac{1}{2} {H}g_{ij},
\end{align}
where $g^{ij}\defeq \left(g_{ij}\right)^{-1}$. 
\FFF  Important relations are \EEE
\begin{align}\label{eq:AA0H}
	\abs{A}^2 = \abs{A^0}^2+\frac{1}{2}H^2 = 2\abs{A^0}^2+2K,
\end{align}
where $K$ denotes the Gauss curvature.
Consequently, using \eqref{eq:WvstildeW}, we find
\begin{align}\label{eq:A^2GaussBonnet}
	\int_{\Sigma} \abs{A}^2\diff \mu  = {\CalW_0}(f)+2\CalW(f)=4{\CalW}(f)-8\pi + 8\pi g.
\end{align}
The Levi-Civita connection $\nabla=\nabla_f$ induced by the metric $g_f$ extends uniquely to a connection on tensors, which we also denote by $\nabla$. For an orthonormal basis $\{e_1, e_2\}$ of the tangent space, the Codazzi--Mainardi equations yield
\begin{align}
	\nabla_i H &= (\nabla_j A)(e_i, e_j) = 2(\nabla_j A^{0})(e_i, e_j), \label{eq:nabla H A A^0}
\end{align}
cf.\ \cite[(5)]{KSSI}.

Clearly, potential singularities for the flow \eqref{eq:IsoWF} occur if $\CalV(f)$ becomes zero or if the denominator in \eqref{eq:deflambda} vanishes. Note that in  the latter case $H\equiv const$, thus $f$ is a constant mean curvature immersion. 

\begin{lem}\label{lem:AleksandrovHopf}
	Let $\sigma\in (0,1)$ and let $f\colon \Sigma_g\to\R^3$ be an immersion with $\CalI(f)=\sigma$. Then
	\begin{enumerate}[(i)]
		\item $\CalV(f)\neq 0$;
		\item if $g=0$, i.e.\ $\Sigma_g=\S^2$, or if $f$ is an embedding, then $H\not \equiv const$. In particular, the denominator in \eqref{eq:deflambda} is nonzero.
	\end{enumerate}
\end{lem}
\begin{proof}
		The first statement follows immediately from the definition of $\CalI$.
		For (ii), we assume by contradiction that $H\equiv const$, so $f\colon\Sigma_g\to\R^3$ is an immersion with constant mean curvature. If $\Sigma_g=\S^2$, then $f$ has to parametrize a round sphere by a result of Hopf  \cite[Theorem 2.1, Chapter VI]{Hopf}. In the second case, $f$ 
		has to parametrize a round sphere by the famous theorem of Aleksandrov \cite{Aleksandrov}. In both cases this contradicts $\sigma\neq 1$. 
\end{proof}
Despite its geometric degeneracy, \eqref{eq:IsoWF} is still a parabolic equation. Thus, starting with a smooth non-singular initial datum, it is possible to prove the following short-time existence result in similar fashion as it is outlined in \FFF \cite[Chapter 4, Proposition 2.1]{Rupp_2022}, \EEE after observing that we can integrate by parts in \eqref{eq:deflambda} so that the numerator of the Lagrange-multiplier contains no second order derivatives of $A$ any more.
\begin{prop}\label{prop:STE}
	Let $f_0\colon\Sigma_g\to\R^3$ be a smooth immersion with $H_{f_0} \not\equiv const$ and $\CalI(f_0)=\sigma\in (0,1)$. Then there exist $T\in (0,\infty]$ and a unique, non-extendable  $\sigma$-isoperimetric Willmore flow \linebreak $f\colon[0,T)\times \Sigma_g\to\R^3$  with initial datum $f(0)\defeq f(0,\cdot)=f_0$.
\end{prop}
If $\Sigma_g = \S^2$, assumption $H_{f_0}\not \equiv const$ in \Cref{prop:STE} follows from $\sigma\in (0,1)$ by \Cref{lem:AleksandrovHopf} (ii).


 \subsection{Evolution of geometric quantities}
	 In this subsection, we will briefly review the variations of the relevant geometric quantities and energies.
	 
	 \begin{lem}[{\cite[Lemma 2.3]{RuppVolumePreserving}}]\label{lem:geometricEvolutions}
	 	Let $f\colon [0,T)\times \Sigma_g\to \R^3$ be a smooth family of {immersions} with normal velocity $\partial_t f =\xi\nu$. For an  orthonormal basis $\{e_1, e_2\}$ of the tangent space, the geometric quantities induced by $f$ satisfy
	 	\begin{align}
	 		\partial_t (\diff \mu) &= -H\xi \diff \mu,\label{eq:dtdmu}\\
	 		\partial_t  H &= \Delta \xi+ |A|^2\xi,\label{eq:dtH}\\
	 		(\partial_t  A)(e_i,e_j) &= \nabla^2_{ij} \xi - A_{ik}A_{kj}\xi.\label{eq:dtA}\\
	 	\end{align}
	 \end{lem}
	 As a consequence, we have the following first variation identities, cf.\ \cite[Lemma 2.4]{RuppVolumePreserving}.
	 
	 \begin{prop}\label{prop:L^2grads}
	 	Let $f\colon\Sigma_g\to\R^3$ be an immersion and let $\varphi\in C^{\infty}(\Sigma_g;\R^3)$. Then 
	 	we have
	 	\begin{align}
	 		\CalW_0^{\prime}(f)\FFF\varphi\EEE &=\langle \nabla\CalW_0(f), \varphi\rangle_{L^2(\diff \mu)}=\int \left\langle(\Delta H + |A^0|^2H)\nu, \varphi\right\rangle \diff\mu, \label{eq:1st vari willmore}\\
	 		\CalA^{\prime}(f)\varphi &=\langle \nabla\CalA(f), \varphi\rangle_{L^2(\diff \mu)} = -\int \langle H \nu, \varphi\rangle\diff \mu, \label{eq:1st vari area} \\
	 		\CalV^{\prime}(f)\varphi &= \langle \nabla\CalV(f), \varphi\rangle_{L^2(\diff \mu)} = -\int\langle \nu, \varphi\rangle\diff \mu \label{eq:1st vari vol}.
	 	\end{align}
	 	Moreover, if $\CalI(f)>0$, we have
	 	\begin{align}
	 		\CalI^{\prime}(f)\varphi &=  \langle \nabla\CalI(f), \varphi\rangle_{L^2(\diff \mu)} =
	 		\sigma \int\left\langle\frac{3}{\CalA(f)} H\nu - \frac{2}{\CalV(f)}\nu, \varphi\right\rangle\diff \mu. \label{eq:L^2gradI}
	 	\end{align}
	 \end{prop}
 
	 \begin{proof}
	 	Since $\CalW_0, \CalA$ and $\CalV$ are invariant under orientation-preserving diffeomorphisms of $\Sigma_g$, we only need to consider normal variations, as any tangential variation corresponds to a suitable orientation-preserving family of  reparametrizations (see for instance \cite[Theorem 17.8]{Lee}), which leaves the quantities unchanged.
	 	
	 	The variation of $\CalA$ then follows immediately from \eqref{eq:dtdmu}. For $\CalW_0$ and $\CalV$ consider \cite[Lemma 2.4]{RuppVolumePreserving}, for instance. The variation of $\CalI$ then follows. 
	 \end{proof}
 	
	The scaling behavior of the energies yields the following important identities.
 	\begin{lem}\label{rem:scaling}
 		Let $f\colon \Sigma_g\to\R^3$ be an immersion. Then we have
 		\begin{align}
 			\int \langle (\Delta H + \abs{A^0}^2H)\nu, f\rangle\diff\mu = 0,\; 
 			-\int \langle H\nu, f\rangle\diff\mu = 2\CalA(f), \;
 			-\int \langle \nu, f\rangle\diff\mu = 3\CalV(f).
 		\end{align}
 	\end{lem}
 	\begin{proof}
 		By the scaling invariance of the Willmore energy, we find
 		\begin{align}
 			\langle \nabla \CalW_0(f), f\rangle_{L^2(\diff\mu)}= \left.\frac{\diff}{\diff \alpha}\right\vert_{\alpha=1}\CalW(\alpha f) = 0,
 		\end{align}
 		so \Cref{prop:L^2grads} yields the claim. For $\CalA$ and $\CalV$ we may proceed similarly, using the scaling behavior $\CalA(\alpha f)=\alpha^2\CalA(f), \CalV(\alpha f)=\alpha^3\CalV(f)$ for all $f\colon \Sigma_g\to\R^3, \alpha>0$.
 	\end{proof}
 	

 	This yields that Helfrich immersions are precisely the stationary solutions of  \eqref{eq:IsoWF}--\eqref{eq:deflambda}.
 	\begin{lem}\label{lem:Helfrich vs stationary}
 		Let $f\colon \Sigma_g\to\R^3$ be an immersion with $\CalI(f)=\sigma\in (0,1)$ and $H_f\not\equiv const$. Then $f$ is a Helfrich immersion
 		if and only if it is a stationary solution to the $\sigma$-isoperimetric Willmore flow.
 	\end{lem}
 	\begin{proof}
 		The ``if'' part of the statement is immediate. Suppose $f$ is a Helfrich immersion.
 		We multiply \eqref{eq:Helfrich eq} with $\langle f, \nu\rangle$, integrate and use \FFF\Cref{rem:scaling}\EEE~to conclude
 		\begin{align}\label{eq:2lA+3lV=0}
 			2\lambda_1 \CalA(f) + 3\lambda_2\CalV(f) = 0.
 		\end{align}
 		By \Cref{lem:AleksandrovHopf}(i) we have  $\CalV(f)\neq 0$.
 		Hence, with $\lambda \defeq -\frac{\lambda_2 \CalV(f)}{2}$, Equation \eqref{eq:Helfrich eq} reads
 		\begin{align}\label{eq:stationary}
 			\Delta H + \abs{A^0}^2 H - \lambda\left(\frac{3}{\CalA(f)}H-\frac{2}{\CalV(f)}\right)=0.
 		\end{align}
 		We have $\nabla\CalI(f)\neq 0$ by \Cref{prop:L^2grads}, so by testing \eqref{eq:stationary} with $\nabla\CalI(f)\nu$ and integrating it follows that $\lambda$ is given as in \eqref{eq:deflambda}, so $f$ is indeed stationary.
 	\end{proof}
	 
	 It is not difficult to see that along a solution of \eqref{eq:IsoWF} with $\CalI(f)>0$, the isoperimetric ratio is indeed preserved, since by \Cref{prop:L^2grads}, \eqref{eq:IsoWF} and \eqref{eq:deflambda} we have
	 \begin{align}
	 	\frac{\diff}{\diff t} \CalI(f) &=  \int\left\langle \nabla\CalI(f), \partial_t f\right\rangle\diff \mu \\
	 	&= \int \left\langle \nabla\CalI(f),- \nabla\CalW_0(f) + \frac{\lambda}{\CalI(f)} \nabla\CalI(f)\right\rangle\diff \mu\\
	 	& = 
	 	-\int \left\langle\nabla\CalI(f), \nabla\CalW_0(f)\right\rangle\diff \mu + \frac{\lambda}{\CalI(f)} \int \abs{\nabla\CalI(f)}^2\diff \mu = 0.\label{eq:dtI=0}
	 \end{align}
	 On the other hand, the Willmore energy decreases since by \eqref{eq:dtI=0}
	 \begin{align}
	 	\frac{\diff}{\diff t} \CalW_0(f) &= \int \left\langle \left(\Delta H + \abs{A^0}^2 H \right) \nu,\partial_t f\right\rangle\diff \mu = -\int\abs{\partial_t f}^2\diff\mu\leq 0. \label{eq:dtW<=0}
	 \end{align}
 	Equations \eqref{eq:dtI=0} and \eqref{eq:dtW<=0} are the key features in studying the flow \eqref{eq:IsoWF} and of vital importance for our further analysis. We highlight two immediate consequences.
 
	 \begin{remark}\label{rem:strictLyapunov}
	 	\begin{enumerate}[(i)]
	 		\item The computation in \eqref{eq:dtW<=0} implies that $\CalW_0$ is a \emph{strict Lyapunov function} along the flow \eqref{eq:IsoWF}, i.e.\ $\CalW_0$ is strictly decreasing unless $\partial_t f = 0$, so $f$ is stationary (by uniqueness of the solution). By \eqref{eq:WvstildeW}, this also holds for $\CalW$.
	 		\item Since $\CalW$ is monotone, the limit $\lim_{t\nearrow T}\CalW(f(t, \cdot)) \in [\beta_g(\sigma), \CalW(f_0)]$ exists.
	 	\end{enumerate}
	 \end{remark}
 	
 	As \eqref{eq:IsoWF} is a (degenerate) parabolic equation, the scaling behavior in time and space is central in understanding the problem. Therefore, we gather the scaling behavior of some important quantities in the following lemma. 
 	\FFF The powers appearing in the time integrals below will naturally appear later in our energy estimates, see \Cref{prop:curvatureIntegralsEstimate}\EEE.
 	 
 
%
	\begin{lem}\label{lem:scaling}
		Let $\sigma\in (0,1)$, $f\colon [0,T)\times \Sigma_g\to\R^3$ be a $\sigma$-isoperimetric Willmore flow and let $r>0$. Let $\tilde{f}\colon [0,r^{-4}T)\times \Sigma_g \to\R^3$, $\tilde{f}(t,p) \defeq r^{-1}f(r^4t,p)$. Then
		\begin{enumerate}[(i)]
			\item $\tilde{f}$ is a  $\sigma$-isoperimetric Willmore flow;
			\item the Lagrange multiplier $\tilde{\lambda}$ of $\tilde{f}$ satisfies $\tilde{\lambda}(t) = \lambda (r^4t)$;
			\item $\int_0^{T} \frac{\lambda^2}{\CalA(f)^{2}}\diff t = \int_0^{r^{-4}T} \frac{\tilde{\lambda}^{2}}{\CalA(\tilde{f})^{2}}\diff t$ and $\int_0^T \frac{\abs{\lambda}^{\frac{4}{3}}}{\abs{\CalV(f)}^{\frac{4}{3}}}\diff t=\int_0^{r^{-4}T} \frac{\abs{\tilde{\lambda}}^{\frac{4}{3}}}{\abs{\CalV(\tilde{f})}^{\frac{4}{3}}}\diff t$.
		\end{enumerate}
	\end{lem}
	\begin{proof}
		Follows from the scaling behavior of the geometric quantities and a direct calculation.
	\end{proof}
 
\section{Localized energy estimates}\label{sec:energy estimates}
As in \cite[Section 3]{KSSI} and \cite[Section 2.3 and Section 3]{RuppVolumePreserving}, we will start our analysis by localizing the energy decay \eqref{eq:dtW<=0}. The main goal of this section is to show that all derivatives of $A$ can be bounded along the flow, if the \emph{energy concentration} and a suitable time integral involving the Lagrange multiplier are controlled.
Note that at this stage, we do not yet need to assume $\Sigma_g=\S^2$ or any restriction on the initial energy.


\begin{lem}\label{lem:EvolutionOfCurvatureIntegrals}
	Let $\sigma\in (0,1)$ and let $f\colon[0,T)\times\Sigma_g\to\R^3$ be a $\sigma$-isoperimetric Willmore flow. Let $\tilde{\eta} \in C_c^{\infty}(\R^3)$ and define $\eta\defeq  \tilde{\eta}\circ f$. Then we have
	\begin{align}\label{eq:dtIntH^2Equation}
		\begin{split}
			&\partial_t \int\frac{1}{2}H^2\eta\diff \mu + \int \abs{\nabla\CalW_0(f)}^2\eta\diff\mu \\
			&\quad = \frac{3\lambda}{\CalA(f)} \int \Delta H H\eta\diff \mu + \int\left(\frac{3\lambda}{\CalA(f)} H -\frac{2\lambda}{\CalV(f)}\right) \abs{A^0}^2 H \eta \diff \mu\\
			&\qquad  -2 \int \langle\nabla \CalW_0(f), \nu\rangle \langle\nabla H, \nabla \eta\rangle_{g_f} \diff \mu - \int\langle \nabla\CalW_0(f), \nu\rangle H \Delta\eta \diff \mu + \int\frac{1}{2}H^2\partial_t \eta \diff\mu.
		\end{split}
	\end{align}
	and
	\begin{align}
		&\partial_t \int |A^0|^2\eta\diff \mu + \int \abs{\nabla\CalW_0(f)}^2\eta\diff \mu \nonumber\\
		&\qquad= 
		\frac{6\lambda}{\CalA(f)} \int \langle \nabla^{2}H, A^0\rangle_{g_f}\eta\diff \mu + \int \left(\frac{3\lambda}{\CalA(f)}H-\frac{2\lambda}{\CalV(f)}\right) \abs{A^{0}}^{2}H\eta\diff \mu\\
		&\qquad\quad-2\int \langle \nabla\CalW_0(f),\nu\rangle \left( \langle\nabla H,\nabla \eta\rangle_{g_f}+\langle A^{0},\nabla^2 \eta\rangle_{g_f}\right)\diff \mu + \int|A^{0}|^2\partial_t \eta\diff\mu.\nonumber
	\end{align}
\end{lem}
\begin{proof}
	This computation is very similar to \FFF \cite[Chapter 4, Lemma 2.8]{Rupp_2022} (see also \cite[Section 3]{KSSI})\EEE~if one replaces $\lambda\nu$ with $\left(\frac{3\lambda}{\CalA(f)}-\frac{2\lambda}{\CalV(f)}\right)\nu$,
	so we will focus on the differences. We will use a local orthonormal frame $\{e_i(t)\}_{i=1,2}$ for our computations and find
	\begin{align}
		&\partial_t \int\frac{1}{2}H^2\eta\diff \mu + \int \abs{\nabla\CalW_0(f)}^2\eta\diff\mu \\
		&\quad = \int \left(\frac{3\lambda}{\CalA(f)} H - \frac{2\lambda}{\CalV(f)}\right) \Delta H \eta \diff \mu + \int \left(\frac{3\lambda}{\CalA(f)} H - \frac{2\lambda}{\CalV(f)}\right) \abs{A^0}^{2} H \FFF \eta \EEE \diff \mu\\		
		&\quad \quad + \int (2\xi \nabla_i H \nabla_i\eta + H\xi\Delta\eta)\diff \mu  +\int\frac{1}{2}H^2\partial_t \eta \diff\mu, \label{eq:L 3.1 1}
	\end{align}
	writing $\partial_t f = \xi\nu$. Moreover, we have
	\begin{align}
		&\int \left(2\xi\nabla_i H\nabla_i \eta + H\xi\Delta\eta\right)\diff \mu \\
		& = -2 \int \langle\nabla \CalW_0(f), \nu\rangle \nabla_i H \nabla_i \eta \diff \mu + \frac{6 \lambda}{\CalA(f)} \int H \nabla_i H \nabla_i \eta\diff \mu - \frac{4\lambda}{\CalV(f)} \int\nabla_i H \nabla_i \eta\diff \mu \\
		&\quad - \int\langle \nabla\CalW_0(f), \nu\rangle H \Delta\eta \diff \mu + \frac{3 \lambda}{\CalA(f)} \int H^2\Delta \eta \diff \mu - \frac{2 \lambda}{\CalV(f)} \int H\Delta \eta\diff\mu. \label{eq:L 3.1 2}
	\end{align}
	If we carefully combine the terms with $\lambda$ in \eqref{eq:L 3.1 1} and \eqref{eq:L 3.1 2}, the claim follows after integrating by parts, where the terms involving derivatives of $H$ and the factor $\frac{2\lambda}{\CalV(f)}$ cancel.
	
	For the second identity, arguing similarly as in \FFF\cite[Chapter 4, Lemma 2.8]{Rupp_2022}\EEE~we have
	\begin{align}
		\partial_t\left(|A^{0}|^2\diff \mu\right) 
		&= 2\nabla_i(\nabla_j \xi A^{0}(e_i, e_j)) \diff \mu- \nabla_j(\xi \nabla_j H) \diff \mu + \langle\nabla\CalW_0(f), \xi\nu\rangle \diff \mu \\	
		&= 2\nabla_i(\nabla_j \xi A^{0}(e_i, e_j)) \diff \mu- \nabla_j(\xi \nabla_j H) \diff \mu - \abs{\nabla\CalW_0(f)}^2\diff \mu \\
		&\quad + \lambda (\Delta H + |A^0|^2H)\left( \frac{3}{\CalA(f)}H-\frac{2}{\CalV(f)}\right) \diff \mu.
	\end{align}
	Integrating by parts and using \eqref{eq:nabla H A A^0} we conclude
	\begin{align}\label{eq:evolution intA^0eta1}
			&\partial_t \int |A^{0}|^2\eta\diff \mu + \int \abs{\nabla\CalW_0(f)}^2\eta \diff \mu \\
			&\quad = \int \left[ -2 \nabla_j \xi A^{0}_{ij}\nabla_i \eta + \xi \nabla_j H\nabla_j\eta + \lambda (\Delta H + |A^0|^2H)\left( \frac{3}{\CalA(f)}H-\frac{2}{\CalV(f)}\right)\eta \right] \diff \mu  +\int |A^{0}|^2 \partial_t \eta \diff \mu\\
			&\quad = - 2 \int \langle \nabla\CalW_0(f),\nu\rangle A^{0}_{ij}\nabla^2_{ji} \eta \diff \mu + 2\lambda \int \left(\frac{3}{\CalA(f)}H-\frac{2}{\CalV(f)}\right)A^{0}_{ij}\nabla^2_{ij}\eta\diff \mu \\
			& \quad\quad -2\int \langle \nabla\CalW_0(f),\nu\rangle \nabla_i H\nabla_i \eta\diff \mu + 2\lambda\int\left(\frac{3}{\CalA(f)}H-\frac{2}{\CalV(f)}\right) \nabla_i H\nabla_i \eta\diff \mu \\
			&\quad \quad  + \lambda \int(\Delta H + |A^0|^2H)\left( \frac{3}{\CalA(f)}H-\frac{2}{\CalV(f)}\right)\eta  \diff \mu + \int|A^{0}|^2\partial_t \eta\diff\mu.
	\end{align}
	Now, using integration by parts and \eqref{eq:nabla H A A^0} once again, we have
	\begin{align}\label{eq:square}
		\int \left(\frac{3\lambda}{\CalA(f)} H -\frac{2\lambda}{\CalV(f)}\right)\left(2\nabla_i H \nabla_i \eta + 2 A^0_{ij}\nabla_{ij}^{2}\eta + \Delta H \eta\right)\diff \mu= \frac{6\lambda}{\CalA(f)} \int \langle \nabla^{2} H , A^{0}\rangle_{g_f} \FFF \eta\EEE\diff \mu.
	\end{align}
	The claim follows.
%
\end{proof}

We will now carefully estimate the integrals in \Cref{lem:EvolutionOfCurvatureIntegrals}. To this end, we choose a particular class of test functions. Let $\tilde{\gamma}\in C^{\infty}_c(\R^3)$ with $0\leq \tilde{\gamma}\leq 1$ and assume $\norm{D\tilde{\gamma}}{\infty}\leq \Lambda$, $\norm{D^2\tilde{\gamma}}{\infty}\leq \Lambda^2$ for some $\Lambda>0$. Then setting
\begin{align}
	\gamma&\defeq \tilde{\gamma}\circ f \colon[0,T)\times \Sigma_g\to\R\text{ we find }\\
	\abs{\nabla\gamma}&\leq \Lambda \text{ and }|\nabla^2\gamma| \leq \Lambda^2 + |A|\Lambda \label{eq:gamma},
\end{align}
and note that $\gamma(t, \cdot)$ has compact support in \FFF $\Sigma_g$, which is compact\EEE, for all $0\leq t<T$, see also \cite[(3.1)]{RuppVolumePreserving}.

\emph{For the rest of this article, we denote by $C$ a universal constant with $0<C<\infty$ which may change from line to line.}

\begin{lem}\label{lem:curvatureIntegralsEstimate1} Let $\sigma\in (0,1)$, let $f\colon[0,T)\times \Sigma_g\to\R^3$ be a $\sigma$-isoperimetric Willmore flow and let $\gamma$ be as in \eqref{eq:gamma}. Then we have
	\begin{align}
		&\partial_t \int \abs{A}^2\gamma^4 \diff \mu +\frac{3}{2}\int \abs{\nabla \CalW_0(f)}^2\gamma^4\diff \mu \\
		&\quad \leq \frac{C\abs{\lambda}}{\CalA(f)} \left( \int\langle\nabla^2 H, A\rangle_{g_f}\gamma^4\diff \mu + \int\abs{A}^4\gamma^4\diff \mu + \Lambda \int\abs{A}^3\gamma^3\diff \mu\right) \\
		&\quad \quad + \frac{C\abs{\lambda}}{\abs{\CalV(f)}}\left(\int\abs{A}^3\gamma^4\diff \mu + \Lambda\int\abs{A}^2\gamma^3\diff \mu\right) + C\Lambda^4\int_{[\gamma>0]}\abs{A}^2\diff \mu + C\Lambda^2 \int\abs{A}^4\gamma^2\diff \mu.
	\end{align}
\end{lem}

\begin{proof}
	We have $2 \langle \nabla^2 \varphi, A\rangle_{g_f} = 2 \langle \nabla^2\varphi, A^0\rangle_{g_f} + H \Delta \varphi$ for any $\varphi\in C^{\infty}([0,T)\times \Sigma_g)$ by a direct computation in a local orthonormal frame. Hence, using \Cref{lem:EvolutionOfCurvatureIntegrals} and $\abs{A}^2 = \abs{A^0}^2+\frac{1}{2}H^2$, cf.\ \eqref{eq:AA0H}, we find
	\begin{align}
		&\partial_t \int\abs{A}^2\gamma^4\diff \mu + 2 \int\abs{\nabla\CalW_0(f)}^2\gamma^4 \diff \mu\\
		&=\frac{6 \lambda}{\CalA(f)}\left(\int\langle \nabla^2H, A\rangle_{g_f} \gamma^4\diff \mu + \int \abs{A^0}^2H^2\gamma^4\diff \mu\right)  - \frac{4\lambda}{\CalV(f)} \int\abs{A^0}^2H\gamma^4\diff \mu \\
		&\quad - 4 \int \langle \nabla\CalW_0(f),\nu\rangle \langle \nabla H, \nabla\gamma^4 \rangle_{g_f}\diff \mu  - 2\int\langle \nabla\CalW_0(f),\nu \rangle\langle\nabla^2\gamma^4, A\rangle_{g_f}\diff \mu + \int\abs{A}^2  \partial_t \gamma^4\diff \mu.\label{eq:dt A^2 local1}
	\end{align}

	The terms $\int \langle\nabla\CalW_0(f),\nu \rangle \langle \nabla H, \nabla\gamma^4 \rangle_{g_f}\diff \mu$ and $\int \langle \nabla\CalW_0(f),\nu\rangle \langle\nabla^2\gamma^4, A\rangle_{g_f}\diff \mu$ can be estimated as in \cite[Lemma 3.2]{KSSI}.
	Since $\partial_t \gamma^4 = 4\gamma^3 \nabla\tilde{\gamma}\circ f \partial_t f$  we have by \eqref{eq:gamma}
	\begin{align}
		\abs{\partial_t \gamma^4}\leq C \Lambda\gamma^3\left(\abs{\nabla\CalW_0(f)}+\frac{\abs{\lambda}}{\CalA(f)}\abs{A}+\frac{\abs{\lambda}}{\abs{\CalV(f)}}\right).
	\end{align}
	Consequently, we find 
	\begin{align}
		\int\abs{A}^2\partial_t \gamma^4 \diff \mu &\leq \varepsilon \int\abs{\nabla\CalW_0(f)}^2\gamma^{4}\diff \mu + C(\varepsilon)\Lambda^{2} \int \abs{A}^{4}\gamma^{2}\diff \mu \\
		&\quad  +C\Lambda \frac{\abs{\lambda}}{\CalA(f)}\int\abs{A}^{3}\gamma^{3}\diff \mu +C \Lambda\frac{\abs{\lambda}}{\abs{\CalV(f)}}\int \abs{A}^{2}\gamma^{3}\diff\mu.
	\end{align}
	Choosing \FFF $\varepsilon>0$ small enough\EEE, the claim follows from the estimates above.
\end{proof}
Note that on the right hand side of \Cref{lem:curvatureIntegralsEstimate1}, terms involving the Lagrange multiplier multiplied with powers of $A$ up to $4$-th order and even second derivatives of $H$ appear. With the energy, we can only control the $L^2$-norms of $H$ and $A$. In the following \Cref{prop:curvatureIntegralsEstimate} we will close this gap by using higher powers of the Lagrange multiplier, the area and the volume, see also \cite[Proposition 3.3]{RuppVolumePreserving}; these powers behave correctly under rescaling, cf.\ \Cref{lem:scaling}.
We will combine this with the interpolation techniques from \cite{KSGF}, \cite{KSSI} to get control on the local $W^{2,2}$-norm of $A$, in terms of the (localized) Willmore gradient, at least if the $L^2$-norm of $A$ is locally small.

\begin{prop}\label{prop:curvatureIntegralsEstimate} There exist universal constants $\varepsilon_0,c_0, C\in (0,\infty)$ with the following property: Let $\sigma\in(0,1)$, let $f\colon[0,T)\times \Sigma_g\to\R^3$ be a $\sigma$-isoperimetric Willmore flow and let $\gamma$ be as in \eqref{eq:gamma}. If we have
	\begin{align}\label{eq:prop 33 assumption}
		\int_{[\gamma>0]}\abs{A}^2\diff \mu <\varepsilon_0 \quad\text{for some time }t\in[0,T),
	\end{align}
	then at time $t$ we can estimate
	\begin{align}
		&\partial_t \int\abs{A}^2\gamma^4 \diff \mu+c_0\int \left(\abs{\nabla^2 A}^2+\abs{A}^2\abs{\nabla A}^2 + \abs{A}^6\right)\gamma^4\diff \mu\\
		&\leq  C \Lambda^4 \int_{[\gamma>0]}\abs{A}^2\diff \mu  + C   \left( \frac{{\lambda}^2}{\CalA(f)^2} +	\frac{\abs{\lambda}^{\frac{4}{3}}}{\abs{\CalV(f)}^{\frac{4}{3}}} \right) \int \abs{A}^2\gamma^4\diff \mu.
	\end{align}
  Here $\int_{[\gamma_0>0]}\abs{A_0}^2\diff \mu_0 \defeq \int_{[\gamma>0]}\abs{A}^2\diff \mu \vert_{t=0}$.
\end{prop}

\begin{proof}
	Using the assumption and the interpolation inequality in \cite[Proposition 2.6]{KSSI} (see also \cite[Proposition 3.2]{RuppVolumePreserving}), we have at time $t\in[0,T)$
	\begin{align}
		\int \left(\abs{\nabla^2 A}^2+\abs{A}^2\abs{\nabla A}^2 + \abs{A}^6\right)\gamma^4\diff \mu \leq C \int \abs{\nabla \CalW_0(f)}^2\gamma^4\diff \mu + C\Lambda^4 \int_{[\gamma>0]}\abs{A}^2\diff \mu.
	\end{align}
	Consequently, from \Cref{lem:curvatureIntegralsEstimate1}, we find for some $c_0\in (0,\infty)$
	\begin{align}
		&\partial_t \int \abs{A}^2\gamma^4 \diff \mu +2c_0	\int \left(\abs{\nabla^2 A}^2+\abs{A}^2\abs{\nabla A}^2 + \abs{A}^6\right)\gamma^4\diff \mu  \\
		&\quad \leq \frac{C\abs{\lambda}}{\CalA(f)} \left( \int\langle\nabla^2 H, A\rangle_{g_f}\gamma^4\diff \mu + \int\abs{A}^4\gamma^4\diff \mu + \Lambda \int\abs{A}^3\gamma^3\diff \mu\right) \\
		&\quad \quad + \frac{C\abs{\lambda}}{\abs{\CalV(f)}}\left(\int\abs{A}^3\gamma^4\diff \mu + \Lambda\int\abs{A}^2\gamma^3\diff \mu\right)+ C\Lambda^2 \int\abs{A}^4 \gamma^2\diff \mu  + C\Lambda^4\int_{[\gamma>0]}\abs{A}^2\diff \mu.\label{eq: dt A^2 local 2}
	\end{align}
	For the first term on the right hand side of \eqref{eq: dt A^2 local 2}, we infer using Young's inequality
	\begin{align}
		&\frac{\abs{\lambda}}{\CalA(f)} \left( \int\langle\nabla^2 H, A\rangle_{g_f}\gamma^4\diff \mu + \int\abs{A}^4\gamma^4\diff \mu + \Lambda \int\abs{A}^3\gamma^3\diff \mu\right) \\
		&\quad \leq \varepsilon \int\abs{\nabla^2 A}^2\gamma^4\diff \mu + \varepsilon \int\abs{A}^6\gamma^4\diff \mu + \Lambda^2 \int\abs{A}^4\gamma^2\diff \mu  +C(\varepsilon)\frac{\abs{\lambda}^2}{\CalA(f)^2} \int\abs{A}^2\gamma^4\diff \mu.\label{eq:dt A^2 local 3}
	\end{align}
	The second term on the right hand side of \eqref{eq: dt A^2 local 2} can be estimated by using Young's inequality with $p=4$ and $q=\frac{4}{3}$ and $\gamma\leq 1$ to obtain
	\begin{align}
		&\frac{C\abs{\lambda}}{\abs{\CalV(f)}}\left(\int\abs{A}^3\gamma^4\diff \mu + \Lambda\int\abs{A}^2\gamma^3\diff \mu\right) \\
		&\quad \leq \varepsilon \int\abs{A}^6\gamma^4\diff\mu + C(\varepsilon) \frac{\abs{\lambda}^{\frac{4}{3}}}{\abs{\CalV(f)}^{\frac{4}{3}}} \int\abs{A}^2\gamma^4\diff \mu + C\Lambda^{4}\int_{[\gamma>0]}\abs{A}^2\diff \mu.\label{eq:dt A^2 local 4}
	\end{align}
	Moreover, we have the estimate
	$C\Lambda^2 \int\abs{A}^4 \gamma^2\diff \mu \leq \varepsilon \int\abs{A}^6\gamma^4\diff \mu + C(\varepsilon)\Lambda^4 \int_{[\gamma>0]}\abs{A}^2\diff \mu$ using Young's inequality. Combining this with \eqref{eq: dt A^2 local 2},\eqref{eq:dt A^2 local 3} and \eqref{eq:dt A^2 local 4} and choosing $\varepsilon>0$ sufficiently small, the claim follows.
\end{proof}

Assumption \eqref{eq:prop 33 assumption} means that the second fundamental form is small on the support of $\gamma$. Note that this will only be satisfied locally, since by \eqref{eq:A^2GaussBonnet} we always have $\int\abs{A}^2\diff \mu \in [8\pi, 4 \CalW(f)-8\pi+ 8\pi g]$. We will now study the situation, where \eqref{eq:prop 33 assumption} is satisfied on all balls with a certain radius, yielding a control over the \emph{concentration of the Willmore energy} in $\R^3$. 
 Following \cite{KSRemovability} we introduce the following notation.

\begin{defi}\label{def:curvature concentration function}
	For a smooth family of immersions $f \colon [0,T)\times \Sigma_g\to\R^3$, $t\in [0,T)$, $r>0$, we define the \emph{curvature concentration function} 
	\begin{align}\label{eq:DefConcentrationFunction}
		\kappa(t,r)\defeq \sup_{x\in \R^3} \int_{B_r(x)}	|A|^2\diff \mu.
	\end{align}
\end{defi}
Here and in the rest of this article, we follow the notation of \cite{KSGF}, i.e.\ the integrals over balls $B_r(x)\subset \R^3$ have to be understood over the preimages under $f_t$.

If $\Gamma>1$ denotes the minimal number of balls of radius $1/2$ necessary to cover $B_1(0)\subset \R^3$, then
\begin{align}\label{eq:GammaEstimate1/2}
	\kappa(t, \rho)\leq \Gamma\cdot \kappa(t, \rho/2)\quad \text{for all }0\leq t<T.
\end{align}

We now prove an integrated form of \Cref{prop:curvatureIntegralsEstimate}.

\begin{prop}\label{prop:t IntegratedIntEstimate}
	Let $\varepsilon_0>0$ be as in \Cref{prop:curvatureIntegralsEstimate}.	
	There exist universal constants $\varepsilon_1 \in (0,\varepsilon_0), c_0, C>0$ with the following property: Let $\sigma\in (0,1)$, let $f\colon[0,T)\times \Sigma_g\to\R^3$ be a $\sigma$-isoperimetric Willmore flow and let $\rho>0$ be such that
	\begin{align}\label{eq:prop 43 assumption}
		\kappa(t, \rho) <\varepsilon_1 \quad\text{for all }t\in[0,T).
	\end{align}
	Then for all $x\in \R^3$ and $t\in [0,T)$ we have
	\begin{align}
		&\left.\int_{B_{\rho/2}(x)}\abs{A}^2 \diff \mu\right\vert_{t}+c_0\int_0^t\int_{B_{\rho/2}(x)} \left(\abs{\nabla^2 A}^2+\abs{A}^2\abs{\nabla A}^2 + \abs{A}^6\right)\diff \mu\diff \tau\\
		&\leq \int_{B_{\rho}(x)}\abs{A_0}^2\diff \mu_0+ \frac{C(1+\sigma^{-2})}{\rho^4}\int_0^t  \int_{B_{\rho}(x)}\abs{A}^2\diff \mu \diff \tau  + C   \int_0^t  \frac{{\lambda}^2}{\CalA(f)^2}  \int_{B_{\rho}(x)} \abs{A}^2\diff \mu\diff \tau.
	\end{align}
\end{prop}

\begin{proof}
	Fix $x\in \R^3$. Let $\tilde{\gamma}\in C^{\infty}_c(\R^3)$ be a cutoff function with $\chi_{B_{\rho/2}(x)}\leq \tilde{\gamma} \leq \chi_{B_\rho(x)}$, $\norm{D\tilde{\gamma}}{\infty}\leq \frac{C}{\rho}$ and $\norm{D^2\tilde{\gamma}}{\infty}\leq \frac{C}{\rho^2}$. Therefore, $\gamma\defeq\tilde{\gamma}\circ f$ is as in \eqref{eq:gamma} with $\Lambda=\frac{C}{\rho}$.	
	 Moreover, if we take $\varepsilon_1>0$ small enough, we have the estimate 
	 \begin{align}\label{eq:rho^2<= A}
	 	\rho^2\leq C \CalA(f)
	 \end{align}
	 as a consequence of Simon's monotonicity formula \cite{SimonWillmore}, see also \cite[Lemma 4.1]{RuppVolumePreserving}. Now, since we have $ \abs{\CalV(f)}^{\frac{4}{3}} = \left(\frac{\sigma}{36\pi}\right)^{\frac{2}{3}}\CalA(f)^2$ by \eqref{eq:dtI=0} and \eqref{eq:defI}, we observe 
	\begin{align}	
		\frac{\abs{\lambda}^{\frac{4}{3}}}{\abs{\CalV(f)}^{\frac{4}{3}}}= \frac{\left(36\pi\right)^{\frac{2}{3}}}{\CalA(f)^2}\abs{\lambda}^{\frac{4}{3}}\sigma^{-\frac{2}{3}}\leq \frac{C(\lambda^2+ \sigma^{-2})}{\CalA(f)^2} \leq C\left(\frac{\lambda^2}{\CalA(f)^2}+\frac{\sigma^{-2}}{\rho^{4}}\right), \label{eq:estimate L^4/3 by L^2 0}
	\end{align}
	where we used Young's inequality (with $p=\frac{3}{2}$ and $q=3$). The statement then immediately follows by integrating \Cref{prop:curvatureIntegralsEstimate} in time.	
\end{proof}

\begin{remark}
	If we directly integrate \Cref{prop:curvatureIntegralsEstimate}, we have to deal with two terms involving $\lambda$, both of whose time integrals behave correctly under parabolic rescaling, cf.\ \Cref{lem:scaling}.
	The estimate \eqref{eq:estimate L^4/3 by L^2 0} above reveals that if \eqref{eq:prop 43 assumption} is satisfied, then it suffices to control merely the $\frac{\lambda^2}{\CalA(f)^2}$-term, since
		\begin{align}	
		\frac{\abs{\lambda}^{\frac{4}{3}}}{\abs{\CalV(f)}^{\frac{4}{3}}}\leq C\left(\frac{\lambda^2}{\CalA(f)^2}+\frac{\sigma^{-2}}{\rho^{4}}\right). \label{eq:estimate L^4/3 by L^2}
	\end{align}
\end{remark}



For the blow-up construction in \Cref{sec:blowup}, we will need the following higher order estimates for the flow in the case of non-concentrated curvature, cf.\ \cite[Theorem 3.5]{KSSI}, \FFF\cite[Proposition 3.5]{RuppVolumePreserving}.\EEE

\begin{prop}\label{prop:high ord small conc}
	Let $\sigma\in(0,1)$ and let $f\colon[0,T)\times \Sigma_g\to\R^3$ be a $\sigma$-isoperimetric Willmore flow. Suppose $\rho>0$ is chosen such that $T\leq T^*\rho^4$ for some $0<T^*<\infty$ and
	\begin{align}
		\kappa(t, \rho)\leq \varepsilon<\varepsilon_1 \quad \text{for all }0\leq t<T,\label{eq:assumed small curv conc}
	\end{align}
	where $\varepsilon_1>0$ is as in \Cref{prop:t IntegratedIntEstimate}. Moreover, assume
	\begin{align}
		\int_0^T \frac{\lambda^2}{\CalA(f)^2}\diff t \leq \bar{L}<\infty. \label{eq:assumed bounds lambda int}
	\end{align}
	 Then for all $t\in (0,T)$, $x\in\R^3$ and $m\in \N_0$ we have the local estimates
	 \begin{align}
	 	\norm{\nabla^m A}{L^2(B_{\rho/8}(x))}&\leq C(m, T^*,\bar{L}, \sigma) \sqrt{\varepsilon}t^{-\frac{m}{4}},\\
	 	\norm{\nabla^m A}{L^{\infty}}&\leq C(m, T^*,\bar{L}, \sigma)\sqrt{\varepsilon}t^{-\frac{m+1}{4}}, \label{eq:highorder local in time estimate}
	 \end{align}
 	and the global bounds
 	\begin{align}\label{eq:small conc global bound D^2A}
 		\norm{\nabla^{m} A}{L^2(\diff \mu)} \leq C(m,T^*,\bar{L}, \sigma) t^{-\frac{m}{4}} \left(\int \abs{A_0}^2\diff \mu_0\right)^{\frac{1}{2}}.
 	\end{align}
\end{prop}

In contrast to \cite[Theorem 3.5]{KSSI} and \FFF \cite[Proposition 3.5]{RuppVolumePreserving}\EEE, we do not only prove local bounds, but also the global $L^2$-control \eqref{eq:small conc global bound D^2A}. Note that the global $L^2$-norms could also be estimated by the $L^{\infty}$-bounds and the area. However, this is disadvantageous since \FFF the \EEE area cannot be controlled along the flow, and in fact is always expected to diverge in the blow-up process, cf.\ \Cref{lem:bounded area compact} below. The necessity for the finer estimates leading to \eqref{eq:small conc global bound D^2A} is why we give full details on the proof here, even though the argument is very similar to \cite[Theorem 3.5]{KSSI}.

\begin{proof}[{Proof of \Cref{prop:high ord small conc}}]
	After parabolic rescaling, cf.\ \Cref{lem:scaling}, we may assume $\rho=1$. Let $x\in \R^3$ and define $K(t)\defeq \int_{B_1(x)} \abs{A}^2\diff \mu$ and $L(t)\defeq \frac{\lambda^2}{\CalA(f)^2}$. 
	 Then, for all $t\in [0, T)$ using that $K\leq \varepsilon<\varepsilon_1$ by \eqref{eq:assumed small curv conc}, we deduce from \Cref{prop:t IntegratedIntEstimate}
	\begin{align}
		&\int_0^t \int_{B_{1/2}(x)} \abs{\nabla^2 A}^2\diff \mu \diff \tau\\
		&\quad \leq C \int_{B_1(x)} \abs{A_0}^2\diff \mu_0 + C(1+\sigma^{-2}) \int_0^t \int_{B_1(x)}\abs{A}^2 \diff \mu \diff \tau + C \int_0^t L \int_{B_1(x)}\abs{A}^2\diff \mu \diff \tau \\
		& \quad \leq C(\sigma) \left(K(0)+\int_0^t (1+L)K\diff \tau\right).\label{eq:53a}
	\end{align}
	Moreover, as $K(t)\leq \varepsilon<\varepsilon_1$ by \eqref{eq:assumed small curv conc} we can interpolate \FFF by combining \cite[Lemma 2.8]{KSSI} and \cite[Lemma 4.2]{KSGF} \EEE 
	to find
	\begin{align}\label{eq:54}
		\int_0^t \norm{A}{L^{\infty}(B_{1/4}(x))}^4 \diff \tau\leq C(\sigma) \varepsilon_1\left(K(0)+\int_0^t (1+L)K\diff \tau\right) \leq C(T^*,\bar{L}, \sigma),
	\end{align}
	where we used the assumptions \eqref{eq:assumed small curv conc}, \eqref{eq:assumed bounds lambda int} and $T\leq T^{*}$ in the last step.
	Thus, defining $a(t) \defeq \norm{A}{L^{\infty}(B_{1/4}(x))}^4$ and using $T\leq T^*$ and \eqref{eq:assumed bounds lambda int}, we have the estimate
	\begin{align}
		\int_0^t (1+L+a)\diff \tau 
		\leq C(T^*, \bar{L}, \sigma). \label{eq:house}
	\end{align}
	Now, we pick $\tilde{\gamma}\in C^{\infty}_c(\R^3)$ with  $\chi_{B_{1/8}(x)}\leq \tilde{\gamma} \leq \chi_{B_{1/4}(x)}$ and $\gamma\defeq \tilde{\gamma}\circ f$. Note that \eqref{eq:gamma} is satisfied with a universal $\Lambda>0$, which we do not keep track of. As in \cite[Theorem 3.5]{KSSI}, we define Lipschitz cutoff functions in time via
	\begin{align}
		\xi_j(t) \defeq \left\lbrace\begin{array}{ll}
			0& \text{for } t\leq (j-1)\frac{T}{m}, \\
			\frac{m}{T}\left(t-(j-1)\frac{T}{m}\right)& \text{for } (j-1)\frac{T}{m}\leq t\leq j\frac{T}{m}, \\
			1 &\text{for } t\geq j\frac{T}{m},
		\end{array}\right.
	\end{align}
	where $m\in \N$ and $0\leq j \leq m$. We also define $\xi_{-1}(t)\defeq 0$ and  $\xi_0(t)\defeq 
	1$ for all $t\in \R$ if $m=0$. We note that $\xi_m(T) =1$ and 
	\begin{align}\label{eq:xi'bound}
		0\leq  \frac{\diff}{\diff t}{\xi}_j \leq \frac{m}{T}\xi_{j-1}\quad \text{for all }j\in \N_0.
	\end{align}
	Furthermore, for $0\leq j\leq m$ we define $E_j(t) \defeq \int |\nabla^{2j} A|^2\gamma^{4j+4}\diff \mu$.
	Then, by \Cref{prop:higher order estimates}, 
	using $\gamma\leq 1$ and \eqref{eq:estimate L^4/3 by L^2} we have
	\begin{align}
		\frac{\diff}{\diff t} E_j(t) + \frac{1}{2} E_{j+1}(t) &\leq C(j,m,\sigma)\Big[\left(1+L(t)+ a(t)\right) E_j(t) + \left(1+L(t) + a(t)\right) K(t)\Big].
	\end{align}
	Therefore, if we define $e_j \defeq \xi_j E_j$ this implies using \eqref{eq:xi'bound} and $\xi_j \leq 1$
	\begin{align}
		\frac{\diff}{\diff t} e_j(t) &\leq \frac{m}{T}\xi_{j-1}(t)E_j(t) + C(j,m,\sigma)\left(1+L(t)+ a(t)\right) e_j(t) \\
		&\quad + C(j,m,\sigma)\left(1+L(t)+ a(t)\right) K(t)	-\frac{1}{2} \xi_j(t) E_{j+1}(t).\label{eq:56}
	\end{align}
	We now claim that for all $0\leq j\leq m$ and $t\in (0,T)$ we have
	\begin{align}\label{eq:HigherOrderLocalizedInduction}
		e_j(t) + \frac{1}{2} \int_0^t \xi_j E_{j+1}\diff s \leq \frac{C(j,m,T^*,\bar{L},\sigma)}{T^j}\left(K(0)+K(t)+\int_0^t(1+L+a)K\diff\tau\right).
	\end{align}	
	We proceed by induction on $j$. For $j=0$ we have $\xi_0 \equiv 1$ on $(0,T)$. Therefore, we clearly have $e_0 =\int |A|^2\gamma^4\diff \mu\leq  K$. Moreover, by \eqref{eq:53a} we find  $$\int_0^tE_{1}(s)\diff s = \int_0^t\int |\nabla^2 A|^2 \gamma^{8}\diff \mu \diff s \leq C(\sigma)\left(K(0)+\int_0^t (1+L)K\diff \tau\right).$$
	
	For $j\geq 1$, integrating \eqref{eq:56} on  $[0,t]$ and using $e_j(0)=0$, we find
	\begin{align}
		&e_j(t) + \frac{1}{2} \int_0^t \xi_jE_{j+1}\diff \tau\\
		&\leq C(j,m,\sigma) \int_0^t \left(1+L + a\right)e_j\diff \tau + C(j,m,\sigma)  \int_0^t\left(1+L + a\right) K\diff \tau + \frac{m}{T} \int_0^t \xi_{j-1} E_j \diff \tau \\
		&\leq C(j,m,T^*,\bar{L},\sigma) \int_0^t \left(1
		+L + a\right)e_j\diff \tau  +\frac{C(j,m,T^*, \bar{L}, \sigma)}{T^{j}}  \left(K(0)+K(t)+\int_0^t(1+L+a)K\diff \tau\right),
	\end{align}
	by the induction hypothesis and since $T\leq T^*$. Using Gronwall's inequality and estimating the exponential term by \eqref{eq:house}, we find
	\begin{align}
		e_j(t)   &\leq  -\frac{1}{2} \int_0^t\xi_jE_{j+1}\diff s + \frac{C(j,m,T^*,\bar{L}, \sigma)}{T^j} \left(K(0)+K(t)+\int_0^t(1+L+a)K\diff \tau\right)\\
		& +  \frac{C(j,m,T^*,\bar{L},\sigma)}{T^j} \int_0^t \left(K(0)+K(\tau)+\int_0^\tau(1+L+a)K\diff s\right)(1+L(\tau)+a(\tau)) \diff \tau.
		\label{eq:Gronwall estimate induction}
	\end{align}
	The estimate \eqref{eq:HigherOrderLocalizedInduction} then follows by using \eqref{eq:house} and estimating the double integral via
	\begin{align}
		&\int_0^t \left(\int_0^\tau (1+L+a)K\diff s\right) (1+L(\tau)+a(\tau))\diff \tau \leq C(T^*, \bar{L}, \sigma) \int_0^t (1+L+a)K\diff s,
	\end{align}
	where we also used \eqref{eq:house} once again. Now, for the local estimates, we evaluate \eqref{eq:HigherOrderLocalizedInduction} at $t=T, j=m$ and use \eqref{eq:assumed small curv conc}, $T\leq T^*$ and \eqref{eq:house}. Recalling $K\leq \varepsilon$ by \eqref{eq:assumed small curv conc}, this yields
	 \begin{align}
	 	\left.\int\abs{\nabla^{2m} A}^2\gamma^{4m+4}\diff \mu\right\vert_{t=T} \leq \frac{C(m,T^{*},\bar{L}, \sigma)}{T^{m}} \varepsilon.\label{eq:57}
	 \end{align}
 	For the global $L^2$-estimate, we observe that \eqref{eq:HigherOrderLocalizedInduction} is linear in $e_j$ and $K$. Hence, as in \cite[Proposition 4.2]{RuppVolumePreserving}, we can sum up the local bounds to get
 	\begin{align}
 		&\int \abs{\nabla^{2m}A}^2\diff \mu \\
 		&\leq \frac{C(m,T^{*},\bar{L}, \sigma)}{T^m} \left(\int\abs{A_0}^2\diff\mu_0 + \int\abs{A}^2\diff \mu + \int_0^T \left(1+L+a\right)\int\abs{A}^2\diff \mu \diff \tau\right)\\
 		&\leq \frac{C(m,T^{*},\bar{L}, \sigma)}{T^m} \int\abs{A_0}^2\diff\mu_0,\label{eq:58}
 	\end{align}
 	where we used \eqref{eq:A^2GaussBonnet}, \eqref{eq:dtW<=0} and \eqref{eq:house}.
 	After renaming $T$ into $t$, \eqref{eq:57} and \eqref{eq:58} are precisely the desired $L^2$-estimate for even orders of derivatives.
 	Exactly as in \cite[Theorem 3.5]{KSSI}, the local and global $L^2$-estimates for $\nabla^{2m+1}A$ follow by interpolation. The $L^{\infty}$-estimate can then be deduced as in \cite[Theorem 3.5]{KSSI} as well.
\end{proof}

\section{Controlling the Lagrange multiplier}\label{sec:lambda}

In this section, we will provide some important estimates for the Lagrange multiplier under the assumption that the initial energy is not too large. In contrast to \cite{RuppVolumePreserving}, the crucial power of $\lambda$ is not of lower order when compared to the left hand side of \Cref{prop:curvatureIntegralsEstimate}. 
Nevertheless, a first immediate feature of the energy regime from \Cref{thm:convergence main} is that we can uniformly bound the denominator of $\lambda$ from below.

\begin{lem}\label{lem: lambda nenner bound}
	Let $\sigma\in (0,1)$ and let $f\colon [0,T)\times \Sigma_g \to\R^3$ be a $\sigma$-isoperimetric Willmore flow with $\CalW(f_0) < \frac{4\pi}{\sigma}$. Then 
	\begin{align}
		\int\Abs{\frac{3}{\CalA(f)} H - \frac{2}{\CalV(f)}}^2\diff \mu \geq \frac{36}{\CalA(f)^2}\left(\sqrt{\frac{4\pi}{\sigma}}- \sqrt{\CalW(f_0)}\right)^2.
	\end{align}
\end{lem}
\begin{proof}
	This follows from the reverse triangle inequality in $L^2(\diff\mu)$, \eqref{eq:dtI=0} and \eqref{eq:dtW<=0}.
\end{proof}

While the scaling techniques from \FFF \cite[Lemma 4.3]{RuppVolumePreserving}\EEE~are not available here, we still get the following key estimate, which gives a control over $\lambda$ by quantities which will be suitably integrable.
\begin{lem}\label{lem:lambda new}
	Let $\sigma\in (0,1)$ and let $f\colon[0,T)\times \Sigma_g\to\R^3$ be a $\sigma$-isoperimetric Willmore flow with $\CalW(f_0)<\frac{4\pi}{\sigma}$. Then, we have
	\begin{align}
		\abs{\lambda} \leq \frac{\sqrt{\CalA(f)}}{6\left(\sqrt{\frac{4\pi}{\sigma}}-\sqrt{\CalW(f)}\right)}\Abs{\int \langle\partial_t f, \nu\rangle\diff \mu +\int\abs{A^0}^2H\diff \mu}.
	\end{align}
\end{lem}
\begin{proof}
	We test the evolution equation \eqref{eq:IsoWF} with the normal $\nu$ and integrate to obtain
	\begin{align}
		\int \langle\partial_t f, \nu\rangle\diff \mu = -\int\abs{A^0}^2H\diff \mu +  \left(\frac{3}{\CalA(f)}\int H\diff \mu - \frac{2}{\CalV(f)}\CalA(f)\right)\lambda,
	\end{align}
	where we used the divergence theorem. We now estimate the prefactor of $\lambda$ by
	\begin{align}
		\Abs{\frac{3}{\CalA(f)}\int H\diff \mu - \frac{2}{\CalV(f)}\CalA(f)} &\geq \frac{2}{\abs{\CalV(f)}}\CalA(f)-\frac{3}{\CalA(f)}\Abs{\int H \diff \mu} \geq  \frac{6}{\sqrt{\CalA(f)}}\left(\sqrt{\frac{4\pi}{\sigma}}-\sqrt{\CalW(f)}\right),
	\end{align}
	using the fact that $\CalI(f)\equiv \sigma$ by \eqref{eq:dtI=0}. By the assumption and \eqref{eq:dtW<=0} this is strictly positive and the claim follows.
\end{proof}

We remark that the existence of  $f_0\colon\Sigma_g\to\R^3$ with $\CalI(f_0)=\sigma\in (0,1)$ satisfying the assumption $\CalW(f_0)<\frac{4\pi}{\sigma}$ is not yet known and --- in general --- not true. For the case $g=0$, this will follow from \Cref{thm:beta0<4pi/sigma} below.  However, for tori we have $\CalW(f_0)\geq 2\pi^2$ by \cite{MarquesNeves}, and hence $2\pi^2 \leq \CalW(f_0)<\frac{4\pi}{\sigma}$ can only hold for $\sigma<\frac{2}{\pi}<1$. On the other hand, for $\sigma\in (0,\frac{1}{2}]$ and arbitrary genus, there exists $f_0$ with $\CalW(f_0)<\frac{4\pi}{\sigma}$ since $\beta_g(\sigma)<8\pi$ by \cite[Theorem 1.2]{KusnerMcGrath}.
We now use \Cref{lem:lambda new} to deduce the time integrability \eqref{eq:assumed bounds lambda int} for $\lambda$ in the case of small curvature concentration, which enables us to bound all derivatives of the second fundamental form by \Cref{prop:high ord small conc}. 

\begin{lem}\label{lem:lambdaBound}
	Let $\sigma\in (0,1)$ and let $f\colon [0,T)\times \Sigma_g\to\R^3$ be a $\sigma$-isoperimetric Willmore flow. Let ${\CalW}(f_0)\leq K<\frac{4\pi}{\sigma}$ and let $\rho>0$ be such that
	\begin{align}
		\kappa(t,\rho)\leq \varepsilon<\varepsilon_1\quad \text{ for all } t\in [0,T),
	\end{align}
	where $\varepsilon_1>0$ is as in \Cref{prop:t IntegratedIntEstimate}. Then for all $\tau \in [0,T)$ we have
	\begin{align}
		&\int_{0}^\tau \frac{\lambda^2}{\CalA(f)^2}\diff t\leq \frac{C}{\left(\sqrt{\frac{4\pi}{\sigma}}-\sqrt{K}\right)^{4}} \left(\CalW(f_0)-\CalW(f(\tau)) +C(\sigma,g) \left(\frac{\tau^{\frac{1}{2}}}{\rho^2}+\frac{\tau}{\rho^4}\right)\right). \label{eq:lambda^2+lambda^4/3 estimate}
	\end{align} 
\end{lem}
Note that by the invariance of the Willmore energy and the isoperimetric ratio, this estimate is preserved under parabolic rescaling, cf.\ \Cref{lem:scaling}. 
\begin{proof}[{Proof of \Cref{lem:lambdaBound}}]
	By the assumption we get the local control from \Cref{prop:t IntegratedIntEstimate}. As in \cite[Proposition 4.2]{RuppVolumePreserving}, we can sum up these local bounds to get the global estimate
	\begin{align}
		\int_0^\tau \int\abs{A}^6\diff \mu\diff t 
		&\leq C \int\abs{A_0}^2\diff \mu_0 + \frac{C(1+\sigma^{-2})}{\rho^{4}}\int_0^\tau \int|A|^2\diff \mu\diff t +C\int_0^\tau  \frac{\lambda^2}{\CalA(f)^2} \int \abs{A}^2\diff \mu\diff t.
	\end{align}
	
	Now, by \eqref{eq:A^2GaussBonnet}, the energy decay \eqref{eq:dtW<=0} and the assumption, we have 
	\begin{align}\label{eq:A^2 int bounded}
		\int\abs{A}^2\diff \mu\leq C(\sigma, g).
	\end{align}
	Thus, we obtain the estimate
	\begin{align}\label{eq:GlobalA^6Estimate}
		\int_0^{\tau}\int \abs{A}^{6}\diff \mu\diff t \leq C\left(\sigma,g\right) \left(1+\frac{\tau}{\rho^4}+ \int_0^\tau\frac{\lambda^2}{\CalA(f)^2} \diff t\right).
	\end{align}
	By \eqref{eq:AA0H} we have  $\abs{A^0}^2\abs{H}\leq C\abs{A}^3$. Therefore, using \Cref{lem:lambda new} we find
	\begin{align}
		&\int_0^{\tau} \frac{\lambda^2}{\CalA(f)^2}\diff t \leq \frac{C}{b^2} \left( \int_0^\tau  \int \abs{\partial_t f}^2\diff \mu \diff t+ \int_0^\tau  \frac{C(\sigma,g)}{\CalA(f)}\int\abs{A}^4\diff \mu\diff t\right),
		\label{eq:Lamba L^2 estimate}
	\end{align}
	by Cauchy--Schwarz and \eqref{eq:A^2 int bounded}, where $b=b(K,\sigma) \defeq \sqrt{\frac{4\pi}{\sigma}}- \sqrt{K}>0$. 
	For the first term in \eqref{eq:Lamba L^2 estimate}, by \eqref{eq:dtW<=0} and \eqref{eq:WvstildeW} we have
	$\int_0^\tau \int \abs{\partial_t f}^2\diff \mu \diff t   = \CalW(f_0)-\CalW(f(\tau)).$
	For the second term in \eqref{eq:Lamba L^2 estimate}, we use \eqref{eq:rho^2<= A}, Cauchy--Schwarz in time and space, \eqref{eq:A^2 int bounded} and then \eqref{eq:GlobalA^6Estimate} to find
	\begin{align}
		\frac{C(\sigma,g)}{b^2}\int_0^\tau \frac{1}{\CalA(f)} \int\abs{A}^4\diff \mu\diff t
		&\leq \frac{C(\sigma,g)\rho^{-2}}{b^2} \left(\int_0^{\tau} \int \abs{A}^6\diff\mu\diff t\right)^{\frac{1}{2}}\left(\int_0^{\tau}\int\abs{A}^2\diff\mu\diff t\right)^{\frac{1}{2}} \\
		&\leq \frac{C(\sigma,g)}{b^2} \frac{\tau^{\frac{1}{2}}}{\rho^2} \left( 1+\frac{\tau^{\frac{1}{2}}}{\rho^2}+ \left(\int_0^\tau\frac{\lambda^2}{\CalA(f)^2}\diff t\right)^{\frac{1}{2}} \right)\\
		&\leq  \frac{C(\delta,\sigma,g)}{b^4}  \left(\frac{\tau^{\frac{1}{2}}}{\rho^2} +\frac{\tau}{\rho^4}\right) + \delta\int_0^{\tau} \frac{\lambda^2}{\CalA(f)^2}\diff t ,\label{eq:A^4GlobalEstimate}
	\end{align}
	for every $\delta>0$ by Young's inequality, and estimating $b=\sqrt{\frac{4\pi}{\sigma}}-\sqrt{K}\leq C(\sigma)$ in the last step. The statement then follows from \eqref{eq:Lamba L^2 estimate} by taking $\delta>0$ sufficiently small.
\end{proof}

%
%

\section{The blow-up and its properties}\label{sec:blowup}
In this section, we will rescale an isoperimetric Willmore flow as we approach the maximal existence time to obtain a limit immersion. Analyzing the properties of this limit will be the keystone in proving our main result, \Cref{thm:convergence main}.

\subsection{A lower bound on the existence time}
As in \cite{KSGF} and \cite{RuppVolumePreserving}, the first step is to prove a lower bound on the existence time of an isoperimetric flow which respects the parabolic rescaling in \Cref{subsec:blowupexistence} below.

To that end, we state a general lifespan result for possible future reference, where the lower bound only depends on the radius of concentration $\rho$, the isoperimetric ratio $\sigma$ and the behavior of the $L^2$-norm of $\frac{\lambda}{\CalA(f)}$ near $t=0$.

\begin{prop}\label{prop:Lifespan general}
Let $\varepsilon_1>0$ be as in \FFF \Cref{prop:t IntegratedIntEstimate}. \EEE	
	There exist universal constants $\bar{\delta}>0$ nd $\bar{\varepsilon}\in (0, \min\{8\pi,\varepsilon_1\})$ with the following property: Let $f_0\colon \Sigma_g\to\R^3$ be an immersion with\linebreak $\CalI(f_0)= \sigma\in (0,1)$, $H_{f_0}\not\equiv const$ and $\CalW(f_0)<\frac{4\pi}{\sigma}$. Let $f$ be the $\sigma$-isoperimetric Willmore flow with initial datum $f_0$. Assume that
	\begin{enumerate}[(a)]
		\item $\kappa(0, \rho) \leq \varepsilon<\bar{\varepsilon}$ for some $\rho>0$;
		\item there exists $\bar{\omega}>0$ with the following property: For any $t_0 \in [0, \min \{T, \rho^4\bar{\omega}\}]$ with $\kappa(t, \rho)<\varepsilon_1$ for all $0\leq t<t_0$, we have $\int_0^{t_0} \frac{\lambda^2}{\CalA^2}\diff t\leq \bar{\delta}$.
	\end{enumerate}
	 Then the maximal existence time of the flow satisfies $T > \hat{c}\rho^4$ for some $\hat{c}=\hat{c}(\sigma, \bar{\omega}) \in (0,1)$ and 
	 \begin{align}\label{eq:lifespan estimate2}
	 	\kappa(t,\rho)\leq \hat{c}^{-1}\varepsilon \quad\text{for all } t\in [0, \hat{c}\rho^4].
	 \end{align}
\end{prop}

%

Note that we always have $\lim_{t\searrow 0}\int_0^t \frac{\lambda^2}{\CalA^2}\diff \tau=0$. The crucial insight here is that only the decay behavior of the $L^2$-norm of $\frac{\lambda}{\CalA}$ under the assumption of small concentration allows control on the existence time in a way which transforms correctly under parabolic rescaling. 


\begin{proof}[Proof of \Cref{prop:Lifespan general}]
	Without loss of generality, we may assume $\rho=1$, otherwise we rescale as in \Cref{lem:scaling}, see also \FFF \cite[Proposition 3.5]{RuppVolumePreserving}. \EEE
	Let $T$ denote the maximal existence time of the flow and let $\bar{\delta}>0$ to be chosen. We define $\bar{\varepsilon}\defeq \min\{\frac{\varepsilon_1}{3\Gamma},8\pi\}$ where $\varepsilon_1>0$ is as in \Cref{prop:curvatureIntegralsEstimate} and $\Gamma>1$ is as in \eqref{eq:GammaEstimate1/2} and set $\kappa(t)\defeq \kappa(t,1)$ for $t\in [0,T)$. By compactness of $f([0,t]\times \Sigma)$ for $t<T$, the supremum in the definition of $\kappa= \kappa(\cdot,1)$ in \eqref{eq:DefConcentrationFunction} is always attained and the function $\kappa\colon[0,T)\to\R$ is continuous with $\kappa(0)\leq \varepsilon<\bar{\varepsilon}$ by (a).
		
	For a parameter $\omega\in (0,\bar{\omega}]$, to be chosen later, we now define
	\begin{align}\label{eq:deft_0}
		t_0 \defeq \sup\left\{ 0\leq t \leq \min\{T, \omega\} \mid \kappa(\tau)\leq 3\Gamma\varepsilon\text{ for all }0\leq \tau <t\right\} \in [0, \min\{T, {\omega}\}].
	\end{align}
	By continuity of $t\mapsto\kappa(t)$ and (a), we have $t_0>0$. For $t\in [0,t_0)$, we have $\kappa(t)\leq 3\Gamma\varepsilon<\varepsilon_1$ by \eqref{eq:deft_0} and the definition of $\bar{\varepsilon}$. Hence, by \Cref{prop:t IntegratedIntEstimate} and assumption (b) we find
	\begin{align}\label{eq:star}
		\int_{B_{1/2}(x)} \abs{A}^2\diff \mu \leq \int_{B_1(x)}\abs{A_0}^2\diff \mu_0 + 3c_1(1+\sigma
		^{-2}) \Gamma\varepsilon t + 3c_1 \Gamma\varepsilon \bar{\delta},
	\end{align}
	for all $0\leq t<t_0$ where $c_1=C$ from \Cref{prop:t IntegratedIntEstimate}. Now, if we choose $\bar{\delta}\defeq (6 c_1\Gamma)^{-1}>0$ and $\omega=\omega(\sigma, \bar{\omega}) =\min \{(6c_1(1+\sigma^{-2})\Gamma)^{-1}, \bar{\omega}\}>0$ we find from \eqref{eq:star}
	\begin{align}
		\int_{B_{1/2}(x)} \abs{A}^2\diff \mu &\leq \int_{B_1(x)}\abs{A_0}^2\diff \mu_0+ \frac{\varepsilon}{2}\omega^{-1}t+ \frac{\varepsilon}{2} \leq 2\varepsilon \quad \text{for all }0\leq t<t_0.
		\label{eq:1/2CurvatureConcentrationEstimate}
	\end{align}
	However, if ${t_0<\min \{T, \omega\}}$, together with \eqref{eq:GammaEstimate1/2}, this implies $\kappa(t)\leq 2\Gamma\varepsilon<\varepsilon_1$ for all $0\leq t<t_0$ by our choice of $\bar{\varepsilon}$. On the other hand, by \eqref{eq:deft_0} and continuity, we must have $\kappa(t_0)=3\Gamma\varepsilon$, a contradiction.  
	
	Consequently, $t_0= \min \{T, \omega\}$ has to hold. 
	Assume $t_0 = T \leq \omega$. Then, as before, from \eqref{eq:1/2CurvatureConcentrationEstimate} and \eqref{eq:GammaEstimate1/2} we find
	\begin{align}\label{eq:estimate A 1.7}
		\kappa(t)\leq 2\Gamma\varepsilon<\varepsilon_1 \quad \text{for all }0\leq t <T=t_0,
	\end{align}
	by the definition of $\bar{\varepsilon}$. As $T\leq \omega\leq \bar{\omega}$ by assumption and $\int_0^{\bar{\omega}}\frac{\lambda^2}{\CalA(f)^{2}}\diff t \leq \bar{\delta}$ by (b), we can apply \Cref{prop:high ord small conc} to conclude that for any $0<\xi<T$ we have
	\begin{align}\label{eq:star 2}
		\norm{\nabla^m A}{\infty} \leq C(m,\xi,\sigma,\bar{\omega}) \text{ for all }m\in \N_0, t\in [\xi, T),
	\end{align}
	and $\norm{\nabla^mA}{L^2} \leq C(m,\xi,\CalW(f_0),\sigma, \bar{\omega})$. Consequently, for all $t\in [\xi,T)$ we can estimate
	\begin{align}\label{eq:lambda/A Linfty estimate}
		\Abs{\frac{\lambda}{\CalA(f)}}^2 \leq \frac{C\left(\Norm{\Delta H}{L^2}^2+\norm{A}{L^{\infty}}^4\Norm{A}{L^2}^2\right)}{\CalA(f)^2\Norm{\frac{3}{\CalA(f)}H-\frac{2}{\CalV(f)}}{L^2}^2}\leq C(\xi, \CalW(f_0),\sigma, \bar{\omega}),
	\end{align}
	using \eqref{eq:deflambda}, Cauchy-Schwarz, \eqref{eq:AA0H} and \Cref{lem: lambda nenner bound}. Similarly, we find
	\begin{align}\label{eq:lambda/V Linfty estimate}
		\Abs{\frac{\lambda}{\CalV(f)}}^2 \leq \frac{C(\sigma)}{\CalA(f)}\int(\abs{\Delta H}^2+\abs{A}^6)\diff\mu \leq C(\sigma) \left(\norm{\Delta H}{\infty}^2 + \norm{A}{\infty}^6\right)\leq C(\xi, \sigma,\CalW(f_0),\bar{\omega}),
	\end{align}
	where we used $\CalI(f)\equiv \sigma$ and \eqref{eq:star 2}.
	Exactly with the same arguments as in \cite[pp. 330 -- 332]{KSGF} (see also \FFF\cite[Chapter 4, Proof of Theorem 1.1 after (5.8)]{Rupp_2022})\EEE, we can deduce that $f(t)$ smoothly converges to a smooth immersion $f(T)$ as $t\nearrow T$. By assumption and the energy decay, we infer from \Cref{lem: lambda nenner bound} that the denominator in \eqref{eq:deflambda} is bounded away from zero for all $t\in [0,T)$, so $f(T)$ is not a constant mean curvature immersion. By \Cref{prop:STE}, we can then restart the flow with initial datum $f(T)$ which contradicts the maximality of $T$. 
	
	Hence, $T> \omega$ has to hold. The estimate \eqref{eq:lifespan estimate2} then follows from \eqref{eq:1/2CurvatureConcentrationEstimate} and \eqref{eq:GammaEstimate1/2} after choosing $\hat{c} = \hat{c}(\sigma,\bar{\omega}) = \min\{\omega,(2\Gamma)^{-1},1\}>0$.
\end{proof}

Together with the integral estimate for the Lagrange multiplier in \Cref{lem:lambdaBound}, this now implies the following
\begin{prop}[Lifespan bound for small energy gap]\label{thm:Lifespan}
	Let $\sigma\in (0,1)$, let $f\colon[0,T)\times \Sigma_g\to\R^3$ be a maximal $\sigma$-isoperimetric Willmore flow such that
	\begin{enumerate}[(i)]
		\item $\CalW(f_0)\leq K<\frac{4\pi}{\sigma}$;
		\item $\kappa(0, \rho)\leq \varepsilon<\bar{\varepsilon}$, where $\bar{\varepsilon}>0$ is as in \Cref{prop:Lifespan general};
		\item $\CalW(f_0)-\lim_{t\nearrow T}\CalW(f(t))\leq \bar{d}$, where $\bar{d}=\bar{d}(K, \sigma,g)>0$.
	\end{enumerate}
	Then the maximal existence time is bounded from below by
	\begin{align}
		T> \hat{c} \rho^4,
	\end{align}
	where $\hat{c}=\hat{c}(K,\sigma, g)$ and for all $0\leq t \leq \hat{c}\rho^4$ we have $\kappa(t, \rho)\leq \hat{c}^{-1}\varepsilon$.
\end{prop}
Note that the limit in (iii) exists due to \Cref{rem:strictLyapunov} (ii).
\begin{proof}[{Proof of \Cref{thm:Lifespan}}]
	We check that the assumptions in \Cref{prop:Lifespan general} are satisfied.
	Let $\bar{\varepsilon}, \bar{\delta}>0$ be as in \Cref{prop:Lifespan general}. Assumption (a) of \Cref{prop:Lifespan general} holds true by assumption (ii). We now verify assumption (b) in \Cref{prop:Lifespan general}. To that end, let $\bar{\omega}>0$ to be chosen and assume that for some $t_0\in [0, \min \{T, \rho^4\bar{\omega}\}]$ we have $\kappa(t, \rho)<\varepsilon_1$ for all $0\leq t<t_0$. By (i), we may apply \Cref{lem:lambdaBound} and use (iii) to find the estimate
	\begin{align}
		\int_0^{t_0} \frac{\lambda^2}{\CalA^2}\diff t&\leq C(K,\sigma,g)\left( \bar{d} + \bar{\omega}^{\frac{1}{2}}+\bar{\omega}\right)\leq 
		 \bar{\delta},
	\end{align}
	if we choose $\bar{d}=\bar{d}(K, \sigma,g)>0$ and $\bar{\omega}=\bar{\omega}(K,\sigma,g)>0$ small enough. The assumptions of \Cref{prop:Lifespan general} are thus fulfilled and the result follows with $\hat{c}=\hat{c}(\sigma, \bar{\omega}) = \hat{c}(K, \sigma,g)$.
\end{proof}

\subsection{Existence of a blow-up}\label{subsec:blowupexistence}
In this section, we will rescale as we approach the maximal existence time $T\in (0,\infty]$ of a $\sigma$-isoperimetric Willmore flow $f\colon [0,T)\times \Sigma_g\to\R^3$ with $\sigma\in (0,1)$. To that end, let $(t_j)_{j\in \N}\subset [0,T), t_j \nearrow T, (r_j)_{j\in \N}\subset(0, \infty), (x_j)_{j\in \N}\subset \R^3$ be arbitrary. By translation invariance and \Cref{lem:scaling} for all $j\in \N$ the flow
\begin{align}\label{eq:blow up flows}
	{f}_j \colon [0, r^{-4}_j(T-t_j))\times \Sigma_g\to\R^3, \\
	{f}_j(t, p) \defeq  r_j^{-1}\left(f(t_j+r_j^4 t, p)-x_j\right)
\end{align}
is also a $\sigma$-isoperimetric Willmore flow with initial datum $f_j(0)=r_j^{-1}(f(t_j, \cdot)-x_j)$ and maximal existence time $r_j^{-4}(T-t_j)$. Throughout this section, we will denote all geometric quantities of the flow $f_j$ with a subscript $j$, such as $A_j, \lambda_j, \kappa_j, \mu_j$ for example. The next lemma guarantees the existence of suitable $t_j, r_j$ and $x_j$.

\begin{lem}\label{lem:existence tjrjxj}
	Let $\sigma\in (0,1)$ and let $f\colon[0,T)\times \Sigma_g\to\R^3$ be a maximal $\sigma$-isoperimetric Willmore flow with $\CalW(f_0)\leq K<\frac{4\pi}{\sigma}$. Let $\hat{c}=\hat{c}(K, \sigma,g)\in (0,1)$ be as in \Cref{thm:Lifespan}.
	Then, there exist sequences $(t_j)_{j\in \N}\subset [0,T), t_j \nearrow T$, $(r_j)_{j\in \N}\subset(0, \infty)$ and $(x_j)_{j\in \N}\subset \R^3$ such that for all $j\in \N$ we have
	\begin{enumerate}[(i)]
		\item $t_j+r_j^4\hat{c}<T$;
		\item $\kappa_j(t, 1) \leq  \bar{\varepsilon}$ for all $t\in [0, \hat{c}]$, where $\bar{\varepsilon}>0$ is as in \Cref{prop:Lifespan general};	
		\item $\inf_{j\in \N}\int_{B_{1}(0)} \abs{A_{f_j(\hat{c}, \cdot)}}^2\diff\mu_{f_j(\hat{c}, \cdot)}>0$.
	\end{enumerate}
\end{lem}
\begin{proof}
	Given any $t\in [0,T)$, with essentially the same arguments as in \FFF\cite[Lemma 6.6]{Rupp_2022}\EEE, one finds a radius $r_t\in (0,\infty)$ such that
	\begin{align}\label{eq:def r_t}
		\alpha\leq\kappa(t, r_t)\leq \hat{c}\bar{\varepsilon}\quad,
	\end{align}
	where $\alpha =\alpha(K, \sigma,g)>0$. One then argues as in \cite[p. 349]{KSRemovability} (see also \cite[Proposition 6.7]{RuppVolumePreserving}), to prove the existence of $t_j \nearrow T$ and $(x_j)_{j\in \N}\subset \R^3$ such that choosing $r_j\defeq r_{t_j}$, we find that (iii) is satisfied.
	
	Now, the flow $f_j$ satisfies $\kappa_j(0, 1)=\kappa(t_j, r_{j}) = \kappa(t_j, r_{t_j})\leq \hat{c}\bar{\varepsilon}<\bar{\varepsilon}$ by \eqref{eq:def r_t} and since $\hat{c}\in (0,1)$. Moreover, by the invariances of the Willmore energy we have $\CalW(f_j(0))\leq K<\frac{4\pi}{\sigma}$ for all $j\in \N$ and
	\begin{align}
		\CalW(f_j(0))-\lim_{t\nearrow r_j^{-4}(T-t_j)}\CalW(f_j(t)) = \CalW(f(t_j))-\lim_{t\nearrow T}\CalW(f(t)) \to 0, \quad\text{as }j\to\infty.
	\end{align}
	Consequently, for $j$ sufficiently large, we can apply \Cref{thm:Lifespan}, to find that the maximal existence time of the flow $f_j$ is bounded from below by $r_j^{-4}(T-t_j)>\hat{c}$ which proves (i) and $\kappa_j(t, 1)\leq \bar{\varepsilon}$ for all $t\in [0, \hat{c}]$ by \eqref{eq:lifespan estimate2} which proves (ii).
\end{proof}

\begin{prop}[Existence and properties of the limit immersion]\label{lem:blowup existence}
	Let $\sigma\in (0,1)$ and suppose \linebreak $f\colon[0,T)\times\Sigma_g\to\R^3$ is a maximal $\sigma$-isoperimetric Willmore flow with $\CalW(f_0)\leq K<\frac{4\pi}{\sigma}$. Let $\hat{c}\in (0,1)$, $t_j \nearrow T$, $(r_j)_{j\in\N}\subset (0,\infty)$ and $(x_j)_{j\in \N}\subset \R^3$ be as in \Cref{lem:existence tjrjxj}. Then, there exists a complete, orientable surface $\hat{\Sigma}\neq \emptyset$ without boundary and a proper immersion $\hat{f}\colon \hat{\Sigma}\to\R^3$ such that, after passing to a subsequence, $r_j\to r\in [0,\infty]$ and 
	\begin{enumerate}[(i)]
		\item  as $j\to\infty$, $\hat{f}_j \defeq f_j(\hat{c}, \cdot)\to \hat{f}$ smoothly on compact subsets of $\R^3$, after reparametrization;
		\item we have $\int_{\overline{B_1(0)}}\abs{\hat{A}}^2\diff \hat{\mu}>0$ and $\CalW(\hat{f}) \leq \CalW(f_0)$;
		\item $\hat{f}$ is a Helfrich immersion, i.e.\ a solution to \eqref{eq:Helfrich eq};
		\item if $\CalA(\hat{f}_j)\to\infty$, then $\hat{f}$ is a Willmore immersion.
	\end{enumerate}
\end{prop}
Any Helfrich immersion $\hat{f}\colon\hat{\Sigma}\to\R^3$ which arises from the process described above is called a \emph{concentration limit}. More precisely, we call $\hat{f}$ a \emph{blow-up} if $r_j\to 0$, a \emph{blow-down} for $r_j \to\infty$ and a \emph{limit under translation} if $r_j\to r\in (0, \infty)$. Note that by \Cref{lem:existence tjrjxj} (i) the last two can only occur if $T=\infty$.

We highlight that \Cref{lem:blowup existence} (iv) is particularly remarkable, since it means that under the assumption of diverging area, the constraint vanishes in the concentration limit, see also \cite[Theorem 6.2]{RuppVolumePreserving} for a similar \emph{rigidity result.} This will be essential in the proof of \Cref{thm:convergence main}.

\begin{proof}[{Proof of \Cref{lem:blowup existence}}]
	After passing to a subsequence, we may assume $r_j\to r$ in $[0,\infty]$. 
	We have $\bar{\varepsilon}<\varepsilon_1$ and $\hat{c}\in (0,1)$ by \Cref{prop:Lifespan general} and hence by \Cref{lem:existence tjrjxj} (ii) we find $\kappa_j(t, 1)<\varepsilon_1$ for all $t\in [0, \hat{c}]$. We may thus use \Cref{lem:lambdaBound} to bound $\int_0^t \frac{\lambda_j^2}{\CalA(f_j)}\diff \tau \leq C(K, \sigma,g)$ for all $t\in [0, \hat{c}]$ and for all $j\in \N$. Consequently, using \Cref{prop:high ord small conc} we conclude that for any $j\in \N$ we have
	\begin{align}\label{eq:Af_j Cinfty bounds}
		\norm{\nabla^m A_j}{\infty} &\leq C(m,K, \sigma,g) t^{-\frac{m+1}{4}},\\
		\norm{\nabla^m A_j}{L^2(\diff \mu_j)}&\leq C(m, K, \sigma,g)t^{-\frac{m}{4}} \quad\text{for }0< t\leq \hat{c}.\label{eq:Af_j global L^2 bounds}
	\end{align}
	Moreover, from Simon's monotonicity formula, cf.\ \cite[(1.3)]{SimonWillmore}, for any $R>0$ we find
	\begin{align}\label{eq:SimonLocalAreabound}
		R^{-2}\mu_j\left(B_R(0)\right) \leq C K<\infty\quad \text{for all }j\in \N.
	\end{align}

	Thus, we may apply the localized version of Langer's compactness theorem (\cite[Theorem 4.2]{KSSI}, see also \FFF\cite[Appendix A]{RuppVolumePreserving}\EEE) to the sequence of immersions $\hat{f}_j \defeq f_j(\hat
	c, \cdot)$.
	After passing to a subsequence, we thus find a proper limit immersion $\hat{f}\colon\hat{\Sigma}\to\R^3$, where $\hat{\Sigma}$ is a complete (possibly empty) surface without boundary, diffeomorphisms $\phi_j\colon\hat{\Sigma}(j)\to U_j$, where $U_j\subset {\Sigma_g}$ are open sets and \linebreak $\hat{\Sigma}(j) = \{ p\in \hat{\Sigma}\mid \abs{\hat{f}(p)}<j\}$, and functions $u_j \in C^{\infty}(\hat{\Sigma}(j);\R^3)$ such that we have
	\begin{align}\label{eq:hat f representation}
		& \hat{f}_j \circ\phi_j = \hat{f} + u_j \quad \text{on }\hat{\Sigma}(j)
	\end{align}
	as well as $\norm{\hat{\nabla}^m u_j}{L^{\infty}(\hat{\Sigma}(j), \hat{g})}\to 0$ as $j\to\infty$ for all $m\in \N_0$, so (i) is proven.
	
	Moreover, sending $j\to\infty$ in \Cref{lem:existence tjrjxj} (iii) and using the smooth convergence on compact subsets, it follows $\int_{\overline{B_1(0)}}\abs{\hat{A}}^2\diff\hat{\mu}>0$ and hence in particular $\hat{\Sigma}\neq \emptyset$. The second statement in (ii) follows from the scaling invariance and the lower semicontinuity of the Willmore energy with respect to smooth convergence on compact subsets of $\R^3$, see \cite[Appendix B]{DMSS20} for instance.
	
	 Let $\xi\in (0, \hat{c})$ be arbitrary. Using \eqref{eq:Af_j global L^2 bounds} and arguing as in \eqref{eq:lambda/A Linfty estimate} and \eqref{eq:lambda/V Linfty estimate}, we find
	\begin{align}\label{eq:lambda/A+lambda/Vestimate}
		\Abs{\frac{\lambda_j}{\CalA(f_j)}} + \Abs{\frac{\lambda_j}{\CalV(f_j)}} \leq C(\xi,K, \sigma,g)\quad \text{for all }t\in [\xi, \hat{c}], j\in \N,
	\end{align}
	which when combined with \eqref{eq:IsoWF} and \eqref{eq:Af_j Cinfty bounds} immediately yields
	\begin{align}\label{eq:dt f_j bound}
		\norm{\partial_t f_j}{\infty}\leq C(\xi, K, \sigma,g)\quad \text{for all }t\in[\xi,\hat{c}], j\in \N.
	\end{align}
	Now, as a consequence of \Cref{lem:higherorder evol}, we find
	\begin{align}
		\Norm{\partial_t \nabla^m A_j}{\infty} &\leq C(m,\xi, K, \sigma, g),\\
		\Norm{\partial_t \nabla^m A_j}{L^2(\diff \mu_j)} &\leq C(m,\xi, K, \sigma, g)\quad \text{for all }t\in [\xi, \hat{c}],m\in \N_0, j\in \N,\label{eq:dt A^m_j bound}
	\end{align}
	using \eqref{eq:Af_j Cinfty bounds}, \eqref{eq:Af_j global L^2 bounds} and \eqref{eq:lambda/A+lambda/Vestimate}.
	Similarly,  using \Cref{lem:H higher order evolution} instead we obtain	
	\begin{align}
		\Norm{\partial_t \nabla^m H_j}{\infty} &\leq C(m,\xi, K, \sigma, g),\\
		\Norm{\partial_t \nabla^m H_j}{L^2(\diff \mu_j)} &\leq C(m,\xi, K, \sigma, g)\quad \text{for all }t\in [\xi, \hat{c}], j\in \N,m\in \N_0.\label{eq:dt H^m_j bound}
	\end{align}
	We will now use this to bound the derivative of the Lagrange multiplier. To that end, we observe that using $\CalI(f_j)\equiv \sigma$ and integration by parts, we find
	\begin{align}\label{eq:lambda/A explicit}
	\frac{\lambda_j}{\CalA(f_j)}&= \frac{-3 \int\abs{\nabla H_j}^2\diff \mu_j+3 \int\abs{A_j^0}^2H_j^2\diff \mu_j-\frac{12\sqrt{\frac{\pi}{\sigma}}}{\CalA(f_j)^{\frac{1}{2}}}\int\abs{A_j^0}^2H_j\diff \mu_j}{\int\Abs{3H_j -\frac{12 \sqrt{\frac{\pi}{\sigma}}}{\CalA(f_j)^{\frac{1}{2}}}}^2\diff \mu_j.}.
	\end{align}
	Note that by \Cref{lem: lambda nenner bound} the denominator is bounded from below by some $C(K, \sigma)>0$. Using \eqref{eq:dtdmu}, \eqref{eq:Af_j Cinfty bounds},\eqref{eq:Af_j global L^2 bounds}, \eqref{eq:dt f_j bound}, \eqref{eq:dt A^m_j bound} and \eqref{eq:dt H^m_j bound}, by direct computation we find
	\begin{align}\label{eq:dt lambda/A bound}
		\Abs{\partial_t \frac{\lambda_j}{\CalA(f_j)}}\leq C(\xi, K, \sigma, g)\quad \text{for all }t\in[0, \hat{c}], j\in \N.
	\end{align}
	Now, \FFF using $\CalI(f_j)\equiv \sigma$\EEE  and \eqref{eq:dtdmu} we infer
	\begin{align}
		\partial_t \frac{\lambda_j}{\CalV(f_j)} &= 
		 \FFF C(\sigma) \left( \partial_t \frac{\lambda_j}{\CalA(f_j)} \CalA(f_j)^{-\frac{1}{2}} - \frac{1}{2} \frac{\lambda_j}{\mathcal{V}(f_j)} \mathcal{A}(f_j)^{-\frac{3}{2}}  \partial_t \mathcal{A}(f_j)\right)\\
		 & = C(\sigma) \left( \partial_t \frac{\lambda_j}{\CalA(f_j)} \CalA(f_j)^{-\frac{1}{2}} + \FFF \frac{1}{2} \frac{\lambda_j}{\CalA(f_j)} \CalA(f_j)^{-\frac{3}{2}} \int H_j \langle \partial_t f_j, \nu_j \rangle\diff \mu_j\right). \EEE
	\end{align}
	Since $\kappa_j(t,1)<\varepsilon_1$ for $t\in[0,\hat{c}]$, we can apply \eqref{eq:rho^2<= A} with $\rho=1$ to obtain $\CalA(f_j)^{-1}\leq C$ and hence using \eqref{eq:dt lambda/A bound}, \eqref{eq:lambda/A+lambda/Vestimate}, \eqref{eq:dt f_j bound} and \eqref{eq:dtW<=0} we have
%
%
	\begin{align}\label{eq:dt lambda/V bound}
		\Abs{\partial_t \frac{\lambda_j}{\CalV(f_j)}} \leq C(\xi, K,\sigma, g)\quad \text{for all }t\in [0, \hat{c}], j\in \N.
	\end{align}
	
	For $j\in\N$, we now define the flows $\tilde{f}_j \defeq f_j\circ\phi_j\defeq f_j(\cdot, \phi_j(\cdot))\colon(0, \hat{c}]\times \hat{\Sigma}(j)\to\R^3$ and observe that they satisfy the $L^{\infty}$-estimates \eqref{eq:Af_j Cinfty bounds} with $\tilde{A}_j$ instead of $A_j$ and the evolution equation
	\begin{align}\label{eq:dt tilde f_j}
		\partial_t \tilde{f}_j = \left[- \Delta \tilde{H}_j - \abs{\tilde{A}^0_j}^2\tilde{H}_j + \lambda_j \left(\frac{3}{\CalA(f_j)} \tilde{H}_j -\frac{2}{\CalV(f_j)}\right)\right]\nu_j\circ \phi_j.
	\end{align}

	As in \FFF \cite[Proof of Theorem 6.2]{RuppVolumePreserving}\EEE, the estimates for $\tilde{f}_j$ together with the $C^1$-estimates for $\frac{\lambda_j}{\CalA(f_j)}$ and $\frac{\lambda_j}{\CalV(f_j)}$ can then be used to deduce that, after passing to a subsequence, the flows $\tilde{f}_j$ converge in $C^1([\xi, \hat{c}];C^m(P;\R^3))$ for all $m\in\N$ and all $P\subset \hat{\Sigma}$ compact to a limit flow $f_{lim}\colon[\xi, \hat{c}]\times \hat{\Sigma}\to\R^3$. Moreover, we may assume $\nu_j \circ \phi_j \to \nu_{lim}$ in $C^1([\xi, \hat{c}];C^m(P;\R^3))$ for all $m\in \N$ and all $P\subset \hat{\Sigma}$ compact, where $\nu_{lim}(t, \cdot)$ is a smooth normal vector field along $f_{lim}(t, \cdot)$ for all $t\in [\xi, \hat{c}]$, as well as \begin{align}
		\frac{\lambda(f_j)}{\CalA(f_j)} \to \lambda_{lim,1}\quad \text{and}\quad \frac{\lambda(f_j)}{\CalV(f_j)} \to \lambda_{lim,2}\quad\text{in } C^0([\xi, \hat{c}];\R) \text{ as }j\to\infty.
	\end{align}
	
	Now, let $P\subset\hat{\Sigma}$ be a fixed compact set and let $j\in \N$ be large enough. Then, using \eqref{eq:dt tilde f_j},  \eqref{eq:dtI=0} and \eqref{eq:dtW<=0}  we find
	\begin{align}
		\int_\xi^{\hat{c}} \int_{P} \abs{\partial_t \tilde{f}_j}^2\diff \tilde{\mu}_j \diff t \leq \int_\xi^{\hat{c}}\int_{\Sigma}\langle -\nabla \CalW_0(f_j) + \lambda_j \sigma^{-1}\nabla\CalI(f_j), \partial_t f_j\rangle \diff \mu_j \diff t =\int_{\xi}^{\hat{c}} \partial_t \CalW_0(f_j)\diff t.
	\end{align}
	In particular, taking $j \to\infty$ and using $\tilde{f}_j \to f_{lim}$ in $C^1([\xi, \hat{c}];C^m(P;\R^3))$ for all $m\in \N$, we find by \Cref{rem:strictLyapunov} (ii)
	\begin{align}
		\int_\xi \int_{P} \abs{\partial_t f_{lim}}\diff \mu_{lim}\diff t\leq \lim_{j\to\infty} \left(\CalW_0(f(t_j+r_j^4\xi))-\CalW_0(f(t_j+r_j^4\hat{c}))\right) =0.
	\end{align}
	Consequently, we have 	$f_{lim} \equiv f_{lim}(\hat{c}, \cdot) =\lim_{j\to\infty} f_j(\hat{c}, \phi_j(\cdot)) = \lim_{j\to\infty} \hat{f}_j\circ\phi_j = \hat{f}$ in $C^m(P;\R^3)$ for all $m\in \N$ and $P\subset \hat{\Sigma}$ compact.
	 
	We observe that $\hat{\nu} \defeq \nu_{lim}(\hat{c}, \cdot)$ is a global and smooth normal vector field along $\hat{f}$ and hence $\hat{\Sigma}$ is orientable. Setting $\hat{\lambda}_1\defeq \lambda_{lim,1}(\hat{c})$, $\hat{\lambda}_2\defeq \lambda_{lim,2}(\hat{c})$  and using \eqref{eq:dt tilde f_j} we find
	\begin{align}
		\left(-\Delta \hat{H}-\abs{\hat{A}^0}^2\hat{H}+ 3\hat{\lambda}_1 \hat{H} - 2\hat{\lambda}_2\right) \hat{\nu}= \lim_{j\to\infty} \partial_t \tilde{f}_j(\hat{c}, \cdot) = \partial_t f_{lim}(\hat{c}, \cdot) = 0\quad \text{on }\hat{\Sigma},
	\end{align}
	so $\hat{f}$ is a Helfrich immersion and (iii) is proven.
	
	For (iv), we now assume $\CalA(\hat{f}_j)\to\infty$ as $j\to\infty$. 
	By \Cref{lem:lambda new}, for all $j\in\N$ and $\xi\in (0, \hat{c})$, we have by Cauchy--Schwarz
	\begin{align}
		&\int_\xi^{\hat{c}} \Abs{\frac{{\lambda}_j}{\CalA({f}_j)}}^2 \diff t \leq C(K, \sigma) \int_\xi^{\hat{c}}  \left(\int \abs{\partial_t f_j}^2 \diff \mu_j + \CalA({f}_j)^{-1}\left(\int\abs{A^0_j}^2\abs{H_j}\diff \mu_j\right)^2\right)\diff t \\
		&\quad \leq C(K, \sigma,g)\Big[ \CalW(f(t_j+r_j^4\xi,\cdot))-\CalW(f(t_j+r_j^4\hat{c}, \cdot))+ \int_\xi^{\hat{c}}\CalA(f_j)^{-1}\diff t \Big],\label{eq:lambda/A nabla H to zero}
	\end{align}
	where we estimated $\int \abs{A_j}^3  \diff \mu_j \leq C(\xi, K, \sigma,g)$ for all $t\in [\xi, \hat{c}]$, using \eqref{eq:Af_j Cinfty bounds} and \eqref{eq:dtW<=0}. Moreover, as a consequence of \eqref{eq:dtdmu} and \eqref{eq:dtW<=0}, for all $t\in [0, \hat{c}]$ we have
	\begin{align}
		\Abs{\CalA(f_j(t, \cdot))-\CalA(f_j(\hat{c}, \cdot))}&\leq \Abs{\int_t^{\hat{c}} \FFF \int \EEE \langle H_j \nu_j, \partial_t f_j \rangle \diff \mu_j \diff \tau}\\
		& \leq 2\hat{c} \CalW(f_0) + \frac{1}{2}\int_0^{\hat{c}}\int\abs{\partial_t f_j}^2\diff \mu_j \diff \tau \leq C(K, \sigma,g),
	\end{align}
	so that $\CalA(f_j(t, \cdot)) \geq \CalA(f_j(\hat{c})) - C(K,\sigma,g)$ for all $t\in [0, \hat{c}]$ and hence the last term on the right hand side of \eqref{eq:lambda/A nabla H to zero} goes to zero as $j\to\infty$. Since  $t_j\nearrow T$, the first term in \eqref{eq:lambda/A nabla H to zero} converges to zero for $j\to\infty$. Consequently
	\begin{align}
		0=\lim_{j\to\infty}\int_{\xi}^{\hat{c}} \Abs{\frac{\lambda_j}{\CalA(f_j)}}^2\diff t = \int_\xi^{\hat{c}} \abs{\lambda_{lim,1}}^2\diff t,
	\end{align}
	so that in particular, $\hat{\lambda}_1 =\lambda_{lim,1}(\hat{c})=0$. Moreover, as $\CalI(f_j)\equiv \sigma$, from \eqref{eq:defI} we obtain 
	\begin{align}
		\Abs{\frac{\lambda_j(\hat{c})}{\CalV(f_j(\hat{c}, \cdot))}} = C(\sigma) \Abs{\frac{\lambda_j(\hat{c})}{\CalA(f_j(\hat{c}, \cdot))}} \CalA(f_j(\hat{c}, \cdot))^{-\frac{1}{2}}\to 0, \quad j\to\infty,
	\end{align}
	so $\hat{\lambda}_2=0$ and thus $\hat{f}$ is a Willmore immersion.
\end{proof}

%

\subsection{The constrained \texorpdfstring{\L ojasiewicz}{Lojasiewicz}--Simon inequality}

\FFF In this section, we establish a {\L}o\-ja\-sie\-wicz--Simon inequality \cite{Loja63,Loja65,MR0727703}. While the unconstrained Willmore energy satisfies such an inequality \cite{CFS09}, the constraint of fixed isoperimetric ratio requires us to prove a refined estimate. To that end,  we rely on the general framework of \emph{constrained} or \emph{refined} {\L}ojasiewicz--Simon inequalities on submanifolds of Banach spaces \cite{Rupp}, see also \cite[Chapter 1, Section 1.2]{Rupp_2022}. \EEE

\begin{thm}[Constrained {\L}ojasiewicz--Simon inequality]\label{thm:Loja}
	Let $f\colon\Sigma_g\to\R^3$ be a Helfrich immersion with $\CalI(f)=\sigma\in (0,1)$ such that $H_f \not \equiv const$. Then, there exist $C, r>0$ and $\theta\in (0, \frac{1}{2}]$ such that for all immersions ${h}\in W^{4,2}(\Sigma_g;\R^3)$ with $\norm{h-f}{W^{4,2}}\leq r$ and $\CalI(h)=\sigma$ we have
	\begin{align}
		\abs{\CalW_0(h)-\CalW_0(f)}^{1-\theta}\leq C\Norm{\nabla{\CalW_0}(h)-\lambda(h)\sigma^{-1}\nabla \CalI(h)}{L^2(\diff \mu_{{h}})}.
	\end{align}
\end{thm}

The proof of this result is very similar to \cite[Section 7.1]{RuppVolumePreserving}, so we will only provide full details on the differences.
Throughout this section we will fix a smooth immersion ${f}\colon\Sigma_g\to\R^3$ with $\CalI(f)=\sigma\in (0,1)$. The \emph{normal Sobolev spaces} along $f$ are 
\begin{align}
	W^{k,2}(\Sigma_g;\R^3)^{\perp}\defeq \{\phi\in W^{k,2}(\Sigma_g;\R^3)\mid P^{\perp}\phi=\phi\},
\end{align}
for $k\in \N_0$, with $L^2(\Sigma_g;\R^3)^{\perp}\defeq W^{0,2}(\Sigma_g;\R^3)^{\perp}$. Here, the $L^2$-inner product always has to be understood with respect to the measure $\mu_f$ and $P^{\perp}$ denotes the normal projection along $f$, given by $P^{\perp} X \defeq \langle X, \nu_f\rangle \nu_f$ for any vector field $X$ along $f$.

Let $r>0$ be sufficiently small and let $$\tilde{U}\defeq \{ \phi\in W^{4,2}(\Sigma_g;\R^3)^{\perp}\mid \phi=u\nu_f \text{ for } \norm{u}{W^{4,2}(\Sigma_g;\R)}<r\}.$$ 
Consider the shifted energies, defined by
\begin{align}
	&W\colon \tilde{U} \to \R, W(\phi)\defeq {\CalW}_0(f+\phi),\\
	&I\colon \tilde{U}\to \R, I(\phi)\defeq \CalI(f+\phi).
\end{align}
Note that this is well-defined, since $f+\phi$ is an immersion for all $\phi\in \tilde{U}$ with $r>0$ small enough, cf.\ \FFF \cite[Lemma 7.5 (i)]{RuppVolumePreserving}\EEE. The first main ingredient towards proving \Cref{thm:Loja} is the analyticity of the energy and the constraint.

\begin{lem}\label{lem:analyticity}
	 For $r>0$ small enough, the following maps are analytic.
	\begin{enumerate}[(i)]
		\item the function $\tilde{U}\to \R, \phi\mapsto W(\phi)$;
		\item the function $\tilde{U}\to L^2(\Sigma_g;\R^3), \phi\mapsto \nabla{\CalW_0}(f+\phi)\rho_{f+\phi}$, where $\diff \mu_{f+\phi} = \rho_{f+\phi}\diff \mu_f$; 
		\item the function $\tilde{U}\to\R, \phi\mapsto I(\phi)$;
		\item the function $\tilde{U}\to L^2(\Sigma_g;\R^3), \phi\mapsto \nabla \CalI(f+\phi)\rho_{f+\phi}$.
	\end{enumerate}
\end{lem}
\begin{proof}
	Statements (i) and (ii) are exactly as in \FFF \cite[Lemma 7.6 (ii) and (iii)]{RuppVolumePreserving}\EEE. By  \FFF\cite[Lemma 7.6 (i) and (iv)]{RuppVolumePreserving}\EEE, the maps $\tilde{U}\to \R, \phi\mapsto \CalA(f+\phi)$ and  $\tilde{U}\to \R, \phi\mapsto \CalV(f+\phi)$ are analytic and hence so is $I$ by definition of the isoperimetric ratio and since $\CalA(f+\phi)>0$ for all $\phi\in \tilde{U}$. For statement (iv) recall from \Cref{prop:L^2grads} that for $\phi\in \tilde{U}$ we have
	\begin{align}
		\nabla\CalI(f+\phi) = \CalI(f+\phi)  \left(\frac{3}{\CalA(f+\phi)} H_{f+\phi}\nu_{f+\phi} - \frac{2}{\CalV(f+\phi)}\nu_{f+\phi}\right).
	\end{align}
	We note that 
	$\tilde{U}\FFF \to \EEE C^0(\Sigma_g;\R^3), \phi \mapsto \nu_{f+\phi}$ is analytic by \FFF\cite[Lemma 7.5 (ii)]{RuppVolumePreserving}\EEE~and
	 $\tilde{U}\FFF \to \EEE L^2(\Sigma_g;\R^3)$, $\phi \mapsto H_{f+\phi}\nu_{f+\phi}$ is analytic by \cite[Lemma 3.2 (iv)]{CFS09}. We have $\CalA(f+\phi)>0$ for all $\phi\in \tilde{U}$ and $\CalV(f+\phi)\neq 0$ by continuity for $r>0$ sufficiently small, since $\CalI(f)=\sigma>0$. This implies (iv).
\end{proof}

We now compute the first and second variations of $W$ and $I$ in terms of their $H$-gradients, see \cite[Section 5]{Rupp}.

\begin{lem}\label{lem:variations}
	Let $H\defeq L^2(\Sigma_g;\R^3)^{\perp}$ and let $r>0$ be sufficiently small. For each $\phi\in \tilde{U}$, the $H$-gradients of $W$ and $I$ are given by
	\begin{align}
		\nabla_{H}W(\phi) &= P^{\perp} \nabla{\CalW_0}(f+\phi)\rho_{f+\phi},\\
		\nabla_{H}I(\phi) &=I(\phi) \left(\frac{3}{\CalA(f+\phi)} H_{f+\phi} - \frac{2}{\CalV(f+\phi)}\right)  P^{\perp}\nu_{f+\phi}\rho_{f+\phi}
		.\label{eq:WV H grads}
	\end{align}
	Moreover, the Fréchet-derivatives of the $H$-gradient maps of $W$ and $I$ at $u=0$ satisfy
	\begin{align}
		\left(\nabla_{H}W\right)^{\prime}(0) &\colon W^{4,2}(\Sigma_g;\R^3)^{\perp}\to L^2(\Sigma_g;\R^3)^{\perp} \quad\text{ is a Fredholm operator with index zero,} \\
		\left(\nabla_H I\right)^{\prime}(0)&\colon W^{4,2}(\Sigma_g;\R^3)^{\perp}\to L^2(\Sigma_g;\R^3)^{\perp} \quad\text{ is compact.}
	\end{align}
\end{lem}

\begin{proof}
	For $\phi,\psi\in \tilde{U}$, we have by \Cref{prop:L^2grads}
	\begin{align}
		\dtzero W(\phi+t\psi) &= \int\langle \nabla{\CalW_0}(f+\phi), \psi\rangle\diff \mu_{f+\phi} \\
		&= \int \left\langle P^{\perp} \nabla{\CalW_0}(f+\phi)\rho_{f+\phi}, \psi\right\rangle \diff \mu_{f} = \left\langle P^{\perp} \nabla{\CalW_0}(f+\phi)\rho_{f+\phi}, \psi\right\rangle_H.
	\end{align}
	Similarly, the statement for $\nabla_H I$ can be shown. The Fredholm property of $(\nabla_HW)^{\prime}(0)$ follows from \eqref{eq:WvstildeW} and \cite[Lemma 3.3 and p. 356]{CFS09}. For the last statement, we observe that for all\linebreak $\phi\in W^{4,2}(\Sigma_g;\R^3)^{\perp}$ we have $\langle \dtzero \nu_{f+t\phi}, \nu_f\rangle =\frac{1}{2}\dtzero \abs{\nu_{f+t\phi}}^2=0$. Hence, using \eqref{eq:dtH} with $\xi= \langle \nu_f, \phi \rangle$ and \Cref{prop:L^2grads} we find 
	\begin{align}
		(\nabla_H I)^{\prime}(0)\phi &= \dtzero \left(I(t\phi) \left(\frac{3}{\CalA(f+t\phi)} H_{f+t\phi} - \frac{2}{\CalV(f+t\phi)}\right)  \rho_{f+t\phi}\right) \nu_f \\
		&= \langle\nabla\CalI(f), \phi\rangle_{L^2(\diff \mu_f)} \left(\frac{3}{\CalA(f)}H_f - \frac{2}{\CalV(f)}\right) \nu_f  - \frac{3\sigma}{\CalA(f)^2} \langle \nabla\CalA(f), \phi\rangle_{L^2(\diff \mu_f)}H_f\nu_f\\
		&\quad + \frac{3\sigma}{\CalA(f)} \left(\Delta \langle \nu_f, \phi\rangle+\abs{A_f}^2\langle \nu_f, \phi\rangle\right)\nu_f + \frac{2\sigma}{\CalV(f)^2} \langle\nabla\CalV(f),\phi\rangle_{L^2(\diff \mu_f)}\nu_f \\
		&\quad - \sigma\left(\frac{3}{\CalA(f)}H_f - \frac{2}{\CalV(f)}\right) H_f\langle \nu_f, \phi\rangle\nu_f,
	\end{align}
	where we used $\CalI(f)=\sigma$ and $\dtzero\rho_{f+t\phi} = -H_f\langle \nu_f, \phi\rangle$ by \eqref{eq:dtdmu}. Since this only involves terms of order two or less in $\phi\in W^{4,2}(\Sigma_g;\R^3)^{\perp}$, the claim follows from the Rellich--Kondrachov Theorem, see for instance \cite[Theorem 2.34]{Aubin}.
\end{proof}

\begin{proof}[Proof of \Cref{thm:Loja}]
	From the assumption $H_f\not\equiv const$, it follows that $\nabla \CalI(f)\neq 0$ and hence $\nabla_H I(0)\neq 0$. 
	As in \FFF\cite[Propsition 7.4]{RuppVolumePreserving}\EEE, we can thus apply \cite[Corollary 5.2]{Rupp} to deduce that \Cref{thm:Loja} is satisfied in \emph{normal directions}, i.e.\ for the functional $W$ with the constraint\linebreak $I=\sigma$. With the methods from \cite[p. 357]{CFS09}, one can then use the invariance of the energies under diffeomorphisms to conclude that \Cref{thm:Loja} holds in all directions.
\end{proof}

As in \cite[Lemma 4.1]{CFS09}, the \L ojasiewicz--Simon inequality yields the following asymptotic stability result, see also \FFF\cite[Lemma 7.9]{RuppVolumePreserving}\EEE~and \cite[Theorem 2.1]{Lengeler} for related results in the context of constrained gradient flows in Hilbert spaces.

\begin{lem}\label{lem:LojaAsymStabil}
		Let $f_W\colon\Sigma_g\to\R^3$ be a Helfrich immersion with $\CalI(f_W)=\sigma\in (0,1)$, $H_{f_W}\not \equiv const$. 
		 Let $k\in \N$, $k\geq 4$, $\delta>0$. Then there exists $\varepsilon=\varepsilon(f_W)>0$ such that if $f\colon [0,T)\times \Sigma_g\to\R^3$ is a $\sigma$-isoperimetric Willmore flow satisfying
		\begin{enumerate}[(i)]
			\item $\norm{f_0-f_W}{C^{k,\alpha}}<\varepsilon$ for some $\alpha>0$;
			\item ${\CalW_0}(f(t))\geq {\CalW_0}(f_W)$ whenever $\norm{f(t)\circ \Phi(t) - f_W}{C^k}\leq \delta$ for diffeomorphisms $\Phi(t) \colon\Sigma_g\to\Sigma_g$;
		\end{enumerate}
		then the flow exists globally, i.e.\ we may take $T=\infty$. Moreover, as $t\to\infty$, it converges smoothly after reparametrization by some diffeomorphisms $\tilde{\Phi}(t)\colon\Sigma_g\to\Sigma_g$ to a Helfrich immersion $f_\infty$, satisfying ${\CalW_0}(f_W)={\CalW_0}(f_\infty)$ and $\norm{f_{\infty}-f_W}{C^k}\leq\delta$.
\end{lem}
Note that by \Cref{lem:Helfrich vs stationary}, $f_W$ above is a stationary solution to \eqref{eq:IsoWF}--\eqref{eq:deflambda}. 
Consequently, the proof of \Cref{lem:LojaAsymStabil} is a straightforward adaptation of \FFF\cite[Lemma 7.9]{RuppVolumePreserving}\EEE, applying our \L ojasiewicz--Simon inequality in \Cref{thm:Loja} and can be safely omitted.
As an important consequence one then finds the following convergence result by following the lines of \cite[Section 5]{CFS09} (see also \cite[Theorem 7.1]{RuppVolumePreserving}), which yields that in the case where $\hat{\Sigma}$ is compact, below the explicit energy threshold no blow-ups or blow-downs may occur.

\begin{thm}\label{thm:compact blowup gives conv}
	Let $\sigma\in (0,1)$, let $f\colon [0,T)\times \Sigma_g\to\R^3$ be a \emph{maximal} $\sigma$-isoperimetric Willmore flow with $\CalW(f_0)<\frac{4\pi}{\sigma}$ and let $\hat{f}\colon\hat{\Sigma}\to\R^3$ be a concentration limit as in \Cref{lem:blowup existence}. If $\hat{\Sigma}$ has a compact component and $H_{\hat{f}}\not\equiv const$, then $\hat{f}$ is a limit under translation. Moreover, the flow exists for all times, i.e.\ $T=\infty$, and, as $t\to\infty$, converges smoothly after reparametrization to a Helfrich immersion $f_\infty$ with $\CalW(f_\infty)=\CalW(\hat{f})$.
\end{thm}

\begin{proof}
		Let $\hat{c}\in (0,1)$, $t_j \nearrow T, (r_j)_{j\in\N}\subset (0,\infty)$ and $(x_j)_{j\in \N} \subset \R^3$ be as in \Cref{lem:blowup existence}. 
		By arguing as in \cite[Lemma 4.3]{KSSI}, we may assume $\hat{\Sigma}=\Sigma_g$ and, by \Cref{lem:blowup existence} (i), we hence have $\hat{f}_j \circ\Phi_j \to\hat{f}$ smoothly on $\Sigma_g$,
	where $\Phi_j \colon\Sigma_g\to\Sigma_g$ are diffeomorphisms. Let $\varepsilon=\varepsilon(\hat{f})>0$ be as in \Cref{lem:LojaAsymStabil}.
	 Then for some fixed $j_0\in \N$ sufficiently large and any $\alpha\in (0,1)$, we may assume $\norm{\hat{f}_{j_0} \circ\Phi_{j_0} - \hat{f}}{C^{4,\alpha}}<\varepsilon$. Moreover, the $\sigma$-isoperimetric Willmore flow
	\begin{align}
		\tilde{f}_{j_0}(t,\cdot)\defeq r_{j_0}^{-1}\left(f(t_{j_0}+r_{j_0}^4t, \cdot)-x_{j_0}\right)\circ \Phi_{j_0}, \quad t\in [0, r_{j_0}^{-4}(T-t_{j_0})),
	\end{align}
	satisfies $\tilde{f}_{j_0}(\hat{c}) =\hat{f}_{j_0} \circ\Phi_{j_0}$. Using \eqref{eq:dtW<=0} and the invariance of the Willmore energy, for any \linebreak $t\in [0, r_{j_0}^{-4}(T-t_{j_0}))$, we find from \Cref{rem:strictLyapunov} (ii)
	\begin{align}
		{\CalW_0}(\tilde{f}_{j_0}(t)) \geq \lim_{s\to r_{j_0}^{-4}(T-t_{j_0})}{\CalW_0}(f(t_{j_0}+r_{j_0}^4s)) = \lim_{s\to T}{\CalW_0}(f(s)) = \lim_{k\to\infty}{\CalW_0}(\hat{f}_k)={\CalW_0}(\hat{f}),
	\end{align}
	where we used the smooth convergence $\hat{f}_k\circ \Phi_k \to \hat{f}$ in the last step. Also note that we have $\CalI(\hat{f}_{j_0}\circ \Phi_{j_0})=\CalI(\hat{f})=\sigma$.
	Thus, by \Cref{lem:LojaAsymStabil} the flow $\tilde{f}_{j_0}$ exists globally and, as $t\to\infty$, converges smoothly after reparametrization by appropriate diffeomorphisms $\tilde\Phi(t)\colon \Sigma_g\to\Sigma_g$ to a Helfrich immersion $f_\infty$ with $\CalW_0(f_\infty)=\CalW_0(\hat{f})$, so $\CalW(f_\infty)=\CalW(\hat{f})$ by \eqref{eq:WvstildeW}. Consequently, $T=\infty$ and for all $t\geq t_{j_0}$ we have
	\begin{align}
		f(t, \Phi_{j_0}\circ \tilde{\Phi}(r_{j_0}^{-4}(t-t_{j_0}))) = r_{j_0} \tilde{f}_{j_0}\left(r_{j_0}^{-4}(t-t_{j_0}), \tilde{\Phi}(r_{j_0}^{-4}(t-t_{j_0}))\right) + x_{j_0} \to r_{j_0} f_{\infty} + x_{j_0}, 
	\end{align}
	as $t\to\infty$ smoothly on $\Sigma_g$. It remains to prove $r\in (0,\infty)$. To that end, we choose times\linebreak $s_k \defeq r_{j_0}^{-4}\left(t_k-t_{j_0}+\hat{c}r_k^4\right)$ for $k\in \N$, such that we find $s_k\to \infty $ as $k\to\infty$, since $t_k\to T=\infty$. We thus obtain
	\begin{align}
		r_{j_0}^{-1}\left(f(t_k+\hat{c}r_k^4, \cdot)-x_{j_0}\right)\circ \Phi_{j_0}\circ \tilde{\Phi}(s_k) = \tilde{f}_{j_0}(s_k)\circ \tilde{\Phi}(s_k)\to f_{\infty}\quad\text{smoothly as }k\to\infty.
	\end{align}
	Consequently, the diameters converge, so
	$d_k\defeq \diam{f(t_k+\hat{c}r_k^4)(\Sigma_g)} \to r_{j_0} \diam f_{\infty}(\Sigma_g)$ as $k\to\infty$,
	whence $\lim_{k\to\infty}d_k\in (0, \infty)$ since $\Sigma_g$ is compact.
	
	 On the other hand, since $\hat{f}_k \circ \Phi_k \to \hat{f}$ smoothly and $r_k\to r\in [0, \infty]$ the limit $$\lim_{k\to\infty}r_k^{-1}d_k=\lim_{k\to\infty}\diam \hat{f}_k=\diam \hat{f}(\Sigma_g)\in (0,\infty)$$ exists, and consequently $r_k\to r\in (0,\infty)$ has to hold. 
\end{proof}

\section{Convergence for spheres}
\label{sec:convergence}
The goal of this section is to prove \Cref{thm:convergence main}. To that end, we want to use the fact that compactness of the concentration limit $\hat{\Sigma}$ yields convergence of the flow by \Cref{thm:compact blowup gives conv}. We first note that the desired compactness follows, if the area along the sequence $\hat{f}_j$ in \Cref{lem:blowup existence} remains bounded.
\begin{lem}\label{lem:bounded area compact}
	Let $\sigma\in (0,1)$, let $f\colon[0,T)\times \Sigma_g\to\R^3$ be a maximal $\sigma$-isoperimetric Willmore flow with $\CalW(f_0)<\frac{4\pi}{\sigma}$ and let $\hat{f}_j$ be as in \Cref{lem:blowup existence}. If $\sup_{j\in \N} \CalA(\hat{f}_j)<\infty$, then $\hat{\Sigma}$ is compact.
\end{lem}
\begin{proof}
	By \Cref{lem:existence tjrjxj} (iii), we have $\hat{f}_j(\Sigma_g)\cap B_1(0)\neq \emptyset$ for all $j\in \N$, where $\hat{f}_j = f_j(\hat{c}, \cdot)$ with $f_j$ as in \eqref{eq:blow up flows}. We now use the diameter bound \cite[Lemma 1.1]{SimonWillmore} to estimate
	$\diam \hat{f}_j(\Sigma_g) \leq C \sqrt{\CalA(\hat{f}_j) \CalW(\hat{f}_j)}$,
	such that using the assumption, the invariances of the Willmore energy and the energy decay \eqref{eq:dtW<=0}, we find $\sup_{j\in \N}\diam \hat{f}_j(\Sigma_g)<\infty$. Consequently, there exists $R\in(0,\infty)$ such that $\hat{f}_j(\Sigma_g)\subset \overline{B_R(0)}$ for all $j\in \N$. Letting $j\to\infty$ and using \Cref{lem:blowup existence} (i) and the definition of smooth convergence on compact subset of $\R^3$, one then easily deduces $\hat{f}(\hat{\Sigma})\subset \overline{B_{R}(0)}$ and then, since $\hat{f}$ is proper, compactness of $\hat{\Sigma}$.
\end{proof}

We will now use \Cref{lem:bounded area compact} and \Cref{lem:blowup existence} to conclude that if the concentration limit is non-compact, then it is not only a Helfrich, but even a Willmore immersion. In the spherical case, the classification in \cite{Bryant1984} and the inversion strategy from \cite{KSRemovability} will then yield a contradiction. Combined with \Cref{thm:compact blowup gives conv}, this will prove our main result.

\begin{proof}[{Proof of \Cref{thm:convergence main}}]
	Since $\Sigma_g=\S^2$ and $\sigma\in (0,1)$, the existence of a unique, non-extendable $\sigma$-isoperimetric Willmore flow with initial datum $f_0$ follows from \Cref{prop:STE} and \Cref{lem:AleksandrovHopf} (ii). Moreover, by \Cref{rem:strictLyapunov} (i), the Willmore energy strictly decreases unless the flow is stationary, in which case global existence and convergence to a Helfrich immersion follow trivially. Thus, we may assume $\CalW(f_0)<\min\{\frac{4\pi}{\sigma}, 8\pi\}$.
	
	Let $\hat{f}\colon\hat{\Sigma}\to\R^3$ be a concentration limit as in \Cref{lem:blowup existence}.	
	If $\hat{\Sigma}$ is compact, we find $\hat{\Sigma}=\S^2$ by \cite[Lemma 4.3]{KSSI} and long-time existence and convergence follow from \Cref{thm:compact blowup gives conv} and the fact that $H_{\hat{f}}\not\equiv const$ by \Cref{lem:AleksandrovHopf} (ii).
	
	For the sake of contradiction, we assume that $\hat{\Sigma}$ is not compact. Then we may assume $\CalA(\hat{f}_j)\to\infty$ by \Cref{lem:bounded area compact}. Consequently, by \Cref{lem:blowup existence} we find that $\hat{f}\colon\hat{\Sigma}\to\R^3$ is a Willmore immersion with $\CalW(\hat{f})\leq \CalW(f_0) <8\pi$. The rest of the argument is as in \cite[Proof of Theorem 1.2]{RuppVolumePreserving}: Denote by $I$ the inversion in a sphere with radius 1 centered at $x_0\not\in \hat{f}(\hat{\Sigma})$ and let $\bar{\Sigma}\defeq I(\hat{f}(\hat{\Sigma}))\cup \{0\}$. Then $\bar{\Sigma}$ is compact. By \cite[Lemma 5.1]{KSRemovability}, $\bar{\Sigma}$ is a smooth Willmore sphere with $\CalW(\bar{\Sigma})<8\pi$ and hence, using Bryant's classification result \cite{Bryant1984}, has to be a round sphere. 
	Thus, $\hat{f}(\hat{\Sigma})$ has to be either a plane or a sphere. Since $\hat{\Sigma}$ is non-compact by assumption, this yields that $\hat{f}$ has to parametrize a plane, a contradiction to \Cref{lem:blowup existence} (ii).
	
	Now the limit immersion $f_\infty\colon\S^2\to\R^3$ satisfies $\CalW(f_\infty)\leq 8\pi$ and solves \eqref{eq:Helfrich eq} for some $\lambda_1,\lambda_2 \in \R$.	
	 It remains to prove $\lambda_1, \lambda_2\neq 0$. 
	Arguing as in the proof of \Cref{lem:Helfrich vs stationary}, we infer 
	 \begin{align}
	 	2\lambda_1 \CalA(f_\infty) + 3\lambda_2 \CalV(f_\infty)=0.
	 \end{align}
	Now,  $\CalV(f_\infty)\neq 0$ by \Cref{lem:AleksandrovHopf} (i) as $\sigma\in (0,1)$ and also $\CalA(f_\infty) >0$. Consequently, if one of $\lambda_1, \lambda_2$ is zero, then so is the other.	
In this case $f_\infty$ is a Willmore sphere with $\CalW(f_\infty)\leq 8\pi$. By Bryant's result \cite{Bryant1984}, it then has to be a round sphere, so $\CalI(f_\infty)=1$, a contradiction and hence $\lambda_1,\lambda_2\neq 0$.
\end{proof}

\Cref{cor} is an immediate consequence of the previous results.
\begin{proof}[Proof of \Cref{cor}]
	By the assumption on the initial energy, \Cref{lem:blowup existence} yields the existence of a suitable blow-up sequence and a concentration limit $\hat{f}$ with the desired properties. 
	If $\hat{f}$ has constant mean curvature $\hat{H}\equiv c$, using \eqref{eq:AA0H} Equation \eqref{eq:Helfrich eq} reads
	\begin{align}
		\frac{1}{2}\hat{H}^3 - 2 \hat{K}\hat{H} - \lambda_1 \hat{H}-\lambda_2 = 0.
	\end{align}
	If $\hat{\Sigma}$ is compact, we conclude $\hat{H}\equiv c\neq 0$ and hence $\hat{K}$ also has to be constant. But then $\hat{f}$ has to parametrize a round sphere (see for instance \cite[Chapter V.1]{Hopf}), a contradiction to $\CalI(\hat{f})=\sigma\in (0,1)$.
	Therefore, statement (a) follows from \Cref{thm:compact blowup gives conv}. If $\hat{\Sigma}$ is not compact, we may assume $\CalA(\hat{f}_j)\to\infty$ by \Cref{lem:bounded area compact} after passing to a subsequence. In this case, $\hat{f}$ is a Willmore immersion by \Cref{lem:blowup existence} (iv), yielding statement (b).
\end{proof}

\section{An upper bound for \texorpdfstring{$\beta_0$}{beta0}}
\label{sec:beta}
In this section, we will prove an upper bound for the minimal Willmore energy of spheres with isoperimetric ratio $\sigma\in (0,1)$.

\begin{thm}\label{thm:beta0<4pi/sigma}
	For every $\sigma\in(0,1)$ we have $\beta_0(\sigma)<\frac{4\pi}{\sigma}$.
\end{thm}
We remark that this estimate becomes sharp for $\sigma\nearrow 1$ since $\beta_0(1)= 4\pi$. On the other hand for $\sigma\in (0,\frac{1}{2}]$, the statement follows since by \cite[Lemma 1]{Schygulla} we have $\beta_0(\sigma)<8\pi$ for all $\sigma\in (0,1)$.
We will prove \Cref{thm:beta0<4pi/sigma} by comparing energy and isoperimetric ratio of an \emph{ellipsoid}. To that end, for $a\in (0,1]$, we define the half-ellipse
\begin{align}
	c_a(t) \defeq (0,a \cos t, \sin t)^T, \quad t\in [-\frac{\pi}{2}, \frac{\pi}{2}],
\end{align}
in the $y$-$z$-plane in $\R^3$.
By rotating the curve $c_a$ around the $z$-axis we obtain a particular type of ellipsoid, a \emph{prolate spheroid}.  More explicitly, we define
\begin{align}
	f_a(t, \theta) = (a \cos t\cos\theta, a \cos t\sin\theta,\sin t)^T \quad \text{for }t\in [-\frac{\pi}{2}, \frac{\pi}{2}], \theta\in [0,2\pi].
\end{align}

Fortunately, its area, volume  and also its Willmore energy can be explicitly computed without the use of elliptic integrals. 

\begin{lem}\label{lem:fa AVW}
	Let $a\in (0,1)$. Then we have
	\begin{enumerate}[(i)]
		\item $\CalV(f_a)=\frac{4\pi}{3} a^2$;
		\item $\CalA(f_a)=2 \pi a \left(a+\frac{\arcsin\sqrt{1-a^2}}{\sqrt{1-a^2}}\right)$;
		\item $\CalW(f_a)=\frac{7\pi}{3} + \frac{2\pi}{3}a^2 + \frac{\pi}{a} \frac{\arcsin\sqrt{1-a^2}}{\sqrt{1-a^2}}$.
	\end{enumerate} 
\end{lem}
\begin{proof}	
	(i) and (ii) are standard formulas, see for instance \cite[Section 4.8]{Zwillinger}. For (iii), we observe that the mean curvature and the surface element of $f_a$ are given by
	\begin{align}
		\diff \mu_{f_a} &= a \cos t \sqrt{a^2\sin^2 t + \cos^2 t}\diff t\diff\theta,\\
		H_{f_a} &= \frac{(1+a^2)\cos^2 t + 2a^2\sin^2 t}{a(\cos^2 t + a^2 \sin^2 t)^{\frac{3}{2}}},
	\end{align}
	by standard formulas for surfaces of revolution, see for instance \cite[Section 3C]{Kuhnel}. In order to compute the Willmore energy, we thus have to evaluate the integral
	\begin{align}
		\CalW(f_a) = \frac{\pi}{2} \int_{-\frac{\pi}{2}}^{\frac{\pi}{2}} \frac{\left((1+a^2)\cos^2 t+2a^2 \sin^2 t\right)^2 \cos t}{a\left(\cos^2 t + a^2 \sin^2 t\right)^{\frac{5}{2}}} \diff t.
	\end{align}
	Substituting $u=\sin t$, this integral can then be explicitly computed yielding (iii).
\end{proof}

\begin{proof}[Proof of \Cref{thm:beta0<4pi/sigma}]
	Clearly, we have $\beta_0(\CalI(f_a))\leq \CalW(f_a)$. Moreover, by \Cref{lem:fa AVW} and a short computation we have
	\begin{align}\label{eq:fa I}
		\CalI(f_a)= \frac{8a}{\left(a+ \frac{\arcsin\sqrt{1-a^2}}{\sqrt{1-a^2}}\right)^3}\quad \text{for all } a\in (0,1).
	\end{align}
	An elementary computation yields $\CalI(f_a)\to 1$ as $a\to 1$ and similarly $\CalI(f_a)\to 0$ as $a\to 0$. Consequently, we have $\{\CalI(f_a)\mid a\in (0,1)\} = (0,1)$ by a continuity argument.
	
	Now, by \eqref{eq:fa I}, we find for all $a\in (0,1)$
	\begin{align}
		\CalW(f_a)-\frac{4\pi}{\CalI(f_a)} = \frac{\pi}{6}\left(14+4a^2+\frac{6 \arcsin\sqrt{1-a^2}}{a\sqrt{1-a^2}}-\frac{3\left(a+\frac{\arcsin\sqrt{1-a^2}}{\sqrt{1-a^2}}\right)^3}{a}\right) = \frac{\pi}{6}F(a),
	\end{align}	
	where the function $F$ is negative for $a\in (0,1)$ by \Cref{lem:F(a)<0} below.
\end{proof}

\begin{lem}\label{lem:F(a)<0}
	The function $F\colon (0,1)\to \R$ defined by
	$$F(a)\defeq 14+4a^2+\frac{6\arcsin\sqrt{1-a^2}}{a\sqrt{1-a^2}}-\frac{3\left(a+\frac{\arcsin\sqrt{1-a^2}}{\sqrt{1-a^2}}\right)^3}{a}$$ satisfies $F(a)<0$ for all $a\in (0,1)$.
\end{lem}
\begin{figure}
	\centering
	\includegraphics[width=0.5\textwidth]{Plot.pdf}
	\caption{The function $F(a)$ in \Cref{lem:F(a)<0}.}
	\label{fig:F(a)<0}
\end{figure}
We will prove \Cref{lem:F(a)<0} in \Cref{sec:F(a)<0} below.
A quick glimpse at the plot of $F$ in \Cref{fig:F(a)<0} illustrates that the statement of \Cref{lem:F(a)<0} is true. However, a rigorous proof seems to be surprisingly difficult, since the function combines trigonometric functions with polynomials and its graph becomes very flat near $F(1)=0$.

\appendix
\section{Proof of \texorpdfstring{{\Cref{lem:F(a)<0}}}{Lemma 7.3}}\label{sec:F(a)<0}

This section is devoted to \FFF proving \EEE \Cref{lem:F(a)<0}. The idea is to make a change of variables, such that the problem is equivalently formulated in terms of a polynomial in $x, \cos x$ and $\sin  x$. Then, we use the power series representation of the Cosine and Sine functions, to reduce the problem to the question if a certain polynomial has roots in a given interval. This last point can then be discussed by studying the Sturm chain of the polynomial.

\begin{proof}[Proof of \Cref{lem:F(a)<0}]
	For $x\in (0,\frac{\pi}{2})$ we consider the function $G(x)\defeq F(\cos x) \sin^3 x \cos x$, so that expanding we find
	$$ G(x) = -3 x^3 - 9 x^2 \cos x \sin x + 6 x \sin^2 x - 9 x \cos^2 x \sin^2 x + 
	14 \cos x \sin^3x +\cos^3 x\sin ^3x .$$
	We observe that $F(a)<0$ for all $a\in (0,1)$ is equivalent to $G(x)<0$ for all $x\in (0,\frac{\pi}{2})$. 
	
	Using the power series expansion of the Cosine and Taylor's theorem with the Lagrange form of the remainder, for any $N\in \N$ we infer
	\begin{align}
		\cos x = \sum_{k=0}^{N} \frac{(-1)^k}{(2k)!} x^{2k} + \frac{1}{(2N+1)!} \cos^{(2N+1)}(\xi) x^{2N+1},
	\end{align}
	for some $\xi=\xi(N)\in (0,x)\subset (0,\frac{\pi}{2})$. An induction argument yields $\cos^{(2N+1)}=(-1)^{N+1}\sin$, so the remainder has a sign, depending on the parity of $N$. Hence, denoting $T_{\cos}^N(x)\defeq \sum_{k=0}^{N}\frac{(-1)^k}{(2k)!}x^{2k}$, we infer
	\begin{align}\label{eq:TaylorCos}
		T_{\cos}^{2n+1}(x)< \cos x < T_{\cos}^{2n}(x) \quad \text{for all }x\in (0,\frac{\pi}{2}), n\in \N_0.
	\end{align}
	By similar arguments, defining $T_{\sin}^{N}(x)\defeq \sum_{k=0}^N \frac{(-1)^k}{(2k+1)!}x^{2k+1}$ and using $\sin^{(2N+2)} = (-1)^{N+1}\sin$ we find
	\begin{align}\label{eq:TaylorSin}
		T_{\sin}^{2n+1}(x)<\sin x <T_{\sin}^{2n}(x)\quad \text{for all }x\in (0,\frac{\pi}{2}), n\in \N_0.
	\end{align} 
	We will now use \eqref{eq:TaylorCos} and \eqref{eq:TaylorSin} to estimate $G$. For $x\in (0,\frac{\pi}{2})$, we have
	\begin{align}
		G(x) &< -3x^3 - 9 x^2 T^{3}_{\cos}(x) T^{3}_{\sin}(x) + 
		6 x T^{2}_{\sin}(x)^2 - 9 x T^{3}_{\cos}(x)^2 T^{3}_{\sin}(x)^2 \\
		&\quad + 14 T^{2}_{\cos}(x) T^{2}_{\sin}(x)^3 + 
		T^{2}_{\cos}(x)^3 T^{2}_{\sin}(x)^3\eqdef k(x).
	\end{align} 
	Now, we observe that $k(x)$ is polynomial of degree $27$. Using \texttt{Mathematica}, we find that this can be simplified to
	\begin{align}
		k(x) = \frac{x^9}{5852528640000} q(x),
	\end{align}
	for the degree $18$ polynomial 
	\begin{align}
		q(x)&=-984711168000 + 660770611200 x^2 - 
		209922048000 x^4 + 40156646400 x^6 \\
		&\quad  - 5069859840 x^8 + 
		437184000 x^{10}- 25717120 x^{12} + 994464 x^{14} - 22944 x^{16} + 
		241 x^{18}.
	\end{align}
	By substituting $x^2=z$, in order to prove $G(x)<0$ for $x\in (0,\frac{\pi}{2})$ it thus suffices to show that 
	\begin{align}
		p(z) &=-984711168000 + 660770611200 z - 
		209922048000 z^2 + 40156646400 z^3\\
		&\quad  - 5069859840 z^4 + 
		437184000 z^5- 25717120 z^6 + 994464 z^7 - 22944 z^8 + 
		241z^9 <0
	\end{align}
	for all $z\in (0, \frac{\pi^2}{4})$. To this end, one may compute the Sturm chain of the polynomial $p$ (see \cite[Theorem 8.8.15]{CohnAlgebra} for instance), to find that there exist no real roots of $p$ in the interval $[0,3]\supset[0, \frac{\pi^2}{4}]$. Consequently, since $p(0)<0$, we find $p(z)<0$ for all $z\in [0, \frac{\pi^2}{4}]$, and hence the claim follows.
\end{proof}
	
\section{Higher order evolution}
	\label{sec:Higherorder}
	
	In this section, we will prove a higher order version of \Cref{prop:curvatureIntegralsEstimate}. 
	To this end, we follow \cite{KSGF,KSSI} and denote by $\phi*\psi$ any multilinear form, depending on $\phi$ and $\psi$ in a universal bilinear way, where $\phi, \psi$ are tensors on on $\Sigma_g$. In particular,
	we have $\abs{\phi*\psi}\leq C\abs{\phi}\abs{\psi}$ for a universal constant $C>0$ and $\nabla(\phi*\psi)=\nabla\phi*\psi+\phi*\nabla\psi$. 
	Moreover, for $m\in \N_0$ and $r\in \N, r\geq 2$ we denote by $P^m_r(A)$ any term of the type
	\begin{align}
		P^m_r(A) = \sum_{i_1+\dots+i_r=m} \nabla^{i_1}A*\dots*\nabla^{i_r}A.
	\end{align}
	In addition, for $r=1$ we extend this definition by denoting by $P^m_1(A)$ any contraction of $\nabla^mA$ with respect to the metric $g$.
	
	With this notation, we observe that along an isoperimetric Willmore flow the covariant derivatives of the second fundamental form $A$ also satisfy a $4$-th order evolution equation. 
	\begin{lem}\label{lem:higherorder evol}
		Let $\sigma\in (0,1)$ and let $f\colon[0,T)\times \Sigma_g\to\R^3$ be a $\sigma$-isoperimetric Willmore flow. \FFF Then \EEE for all $m\in \N_0$ we have
		\begin{align}
			&\partial_t (\nabla^m A)+ \Delta^2 (\nabla^m A) = P^{m+2}_3(A)+P^m_5(A) + \frac{\lambda}{\CalA(f)}( P^{m+2}_1(A)+ P_3^m(A)) + \frac{\lambda}{\CalV(f)} P_2^m(A).
		\end{align}	
	\end{lem}
	\begin{proof}
		We observe $\partial_t f = \xi\nu$ with 
		\begin{align}
			\xi= -\Delta H + P_3^0(A)+ \frac{\lambda}{\CalA(f)} P_1^0(A)-\frac{2\lambda}{\CalV(f)}.\label{eq:dt f star}
		\end{align}
		For $m=0$, we thus find by \eqref{eq:dtA}
		\begin{align}
			\partial_t A &= \nabla^2\xi + A*A*\xi = -\Delta^2 A + P_3^2(A)+ P_5^0(A)+ \frac{\lambda}{\CalA(f)}\left(P^2_1(A)+P^0_3(A)\right)+\frac{\lambda}{\CalV(f)}P_2^0(A),
		\end{align}
		where we used $\nabla^2\Delta H  = \Delta^2 A +P_3^2(A)$ as a consequence of Simons' identity \cite{Simons}. Assume the statement is true for $m\geq 1$. By \cite[Lemma 2.3]{KSGF} with $\phi=\nabla^m A$ and the fact that we are in codimension one, we find
		\begin{align}
			&\partial_t \nabla^{m+1}A + \Delta^2 \nabla^{m+1}A \\
			&\quad = \nabla\left(P^{m+2}_3(A)+P^m_5(A) + \frac{\lambda}{\CalA(f)}( P^{m+2}_1(A)+ P_3^m(A)) + \frac{\lambda}{\CalV(f)} P_2^m(A)\right) \\
			&\qquad + \sum_{i+j+k=3} \nabla^i A *\nabla^j A*\nabla^{k+m} A + A*\nabla\xi *\nabla^m A+ \nabla A*\xi*\nabla^m A\\
			&\quad = P^{m+3}_3(A)+P^{m+1}_5(A) + \frac{\lambda}{\CalA(f)}( P^{m+3}_1(A)+ P_3^{m+1}(A)) + \frac{\lambda}{\CalV(f)} P_2^{m+1}(A),
		\end{align}
		using \eqref{eq:dt f star} in the last step.
	\end{proof}
	With similar computations as above, one finds the following 
	\begin{lem}\label{lem:H higher order evolution}
		Let $\sigma\in (0,1)$ and let $f\colon[0,T)\times \Sigma_g\to\R^3$ be a $\sigma$-isoperimetric Willmore flow. The for all $m\in \N_0$ we have
		\begin{align}
			&\partial_t (\nabla^m H)+ \Delta^2 (\nabla^m H)= P^{m+2}_3(A)+P^m_5(A) + \frac{\lambda}{\CalA(f)}( P^{m+2}_1(A)+ P_3^m(A)) + \frac{\lambda}{\CalV(f)} P_2^m(A).
		\end{align}	
	\end{lem}
	
	We can now prove the following higher order analogue of \Cref{prop:curvatureIntegralsEstimate}.
\begin{prop}\label{prop:higher order estimates}
	Let $\sigma\in (0,1)$, let $f\colon[0,T)\times \Sigma_g\to\R^3$ be a $\sigma$-isoperimetric Willmore flow  and let $\gamma$ be as in \eqref{eq:gamma}. Then for all $m\in \N_0, s\geq 2m+4$ and $\phi=\nabla^m A$ we have
	\begin{align}
		&\frac{\diff}{\diff t} \int\abs{\phi}^2\gamma^s\diff \mu + \frac{1}{2} \int\abs{\nabla^2\phi}^2\gamma^s\diff \mu\\
		&\quad \leq C\left(\frac{\lambda^2}{\CalA(f)^2} + 	\frac{\abs{\lambda}^{\frac{4}{3}}}{\abs{\CalV(f)}^{\frac{4}{3}}}+ \norm{A}{L^{\infty}([\gamma>0])}^4\right) \int\abs{\phi}^2\gamma^s\diff \mu  \\
		&\qquad + C\left(1+\frac{\lambda^2}{\CalA(f)^2} + 	\frac{\abs{\lambda}^{\frac{4}{3}}}{\abs{\CalV(f)}^{\frac{4}{3}}}+ \norm{A}{L^{\infty}([\gamma>0])}^4\right)\int_{[\gamma>0]}\abs{A}^2\diff \mu,
	\end{align}
	where $C=C(s,m,\Lambda)>0$.
\end{prop}
In order to prove \Cref{prop:higher order estimates}, we first recall the following
	\begin{lem}[{\cite[Lemma 3.2]{KSGF}}]\label{lem:KSGFL3.2}
		Let $f\colon [0,T)\times \Sigma_g \to\R^3$ be a normal variation, $\partial_t f = \xi \nu$. Let $\phi$ be a $\ell \choose 0$-tensor satisfying $\partial_t \phi + \Delta^2 \phi = Y$. Then for any $\gamma\in C^2([0,T)\times \Sigma_g)$ and $s\geq 4$ we have
		\begin{align} \label{eq:X}
			&\frac{\diff}{\diff t} \int |\phi|^2\gamma^s\diff \mu + \int \abs{\nabla^2 \phi}^2\gamma^s\diff \mu  \\
			&\leq  2 \int\langle Y, \phi\rangle\gamma^{s}\diff \mu+\int A * \phi * \phi* \xi\gamma^s \diff \mu + \int |\phi|^2 s\gamma^{s-1}\partial_t \gamma \diff \mu  \\
			&\quad + C\int \abs{\phi}^2\gamma^{s-4} \left(\abs{\nabla\gamma}^4 + \gamma^2\abs{\nabla^2\gamma}^2\right)\diff\mu + C\int \abs{\phi}^2 \left(\abs{\nabla A}^2+\abs{A}^4\right)\gamma^s\diff \mu,
		\end{align}
		where $C=C(s)$.
	\end{lem}
	\begin{proof}[{Proof of \Cref{prop:higher order estimates}}]
		In the following, note that the value of $C=C(s,m,\Lambda)$ is allowed to change from line to line. We apply \Cref{lem:KSGFL3.2}  with $Y=\partial_t \phi +\Delta^2 \phi$, $\xi= P^2_1(A)+P_3^0(A)+\frac{\lambda}{\CalA(f)}P_1^0(A)-\frac{2\lambda}{\CalV(f)}$ by \eqref{eq:dt f star} and estimate the terms on the right hand side. Using $\phi= \nabla^m A$ and \Cref{lem:higherorder evol}, we thus have
		\begin{align}
			&2\int \langle Y, \phi\rangle\gamma^s\diff \mu +\int A*\phi*\phi*\xi\gamma^s \diff \mu + C\int|\phi|^2(\abs{\nabla A}^2+\abs{A}^4)\gamma^s\diff\mu \\
			&\quad = \int \left(P^{m+2}_3(A)+P^m_5(A)\right)*\phi\gamma^s\diff \mu \\
			&\quad + \frac{\lambda}{\CalA(f)} \int \left(P_1^{m+2}(A)+P^m_3(A)\right)*\phi\gamma^s\diff\mu + \frac{\lambda}{\CalV(f)}\int P_2^m(A)*\phi\gamma^s\diff \mu.\label{eq:nabla^mApartial ohne dtgamma}
		\end{align}
		Moreover, by \eqref{eq:gamma} we find
		\begin{align}\label{eq:nabla^mApartial_tgamma}
			\int |\phi|^2\gamma^{s-1}\partial_t \gamma\diff \mu &= \int \abs{\phi}^2\gamma^{s-1}\langle D\tilde{\gamma}\circ f, \nu\rangle \left(-{\Delta H} - |A^0|^2 H +  \frac{3\lambda}{\CalA(f)}H-\frac{2\lambda}{\CalV(f)}\right)\diff \mu.
		\end{align}
		{We proceed by estimating all the terms involving $\lambda$ in \eqref{eq:nabla^mApartial ohne dtgamma} and \eqref{eq:nabla^mApartial_tgamma}. For the first $\lambda$-term in \eqref{eq:nabla^mApartial ohne dtgamma}, since $\abs{P_1^{m+2}(A)}\leq C \abs{\nabla^2 \phi}$ 
			we find for every $\varepsilon>0$
			\begin{align}
				\frac{\lambda}{\CalA(f)} \int P^{m+2}_1(A)*\phi\gamma^s\diff \mu \leq \varepsilon\int\abs{\nabla^2\phi}^2\gamma^s\diff \mu + C(\varepsilon) \frac{\lambda^2}{\CalA(f)^2}\int\abs{\phi}^2\gamma^s\diff \mu.
			\end{align}
			For the second term, we use \cite[Corollary 5.5]{KSGF} with $k=m$, $r=4$ to obtain
			\begin{align}
				\frac{\lambda}{\CalA(f)} \int P_3^m(A)*\phi \gamma^s \diff \mu \leq C\Abs{\frac{\lambda}{\CalA(f)}}\norm{A}{L^{\infty}([\gamma>0])}^2\left(\int\abs{\phi}^2\gamma^s\diff\mu+ \int_{[\gamma>0]}\abs{A^2}\diff \mu\right).
			\end{align} 
			The last $\lambda$-term in \eqref{eq:nabla^mApartial ohne dtgamma} can be estimated by \cite[Corollary 5.5]{KSGF} with $k=m$ and $r=3$, yielding
			\begin{align}
				\frac{\lambda}{\CalV(f)} \int P_2^m(A)*\phi\gamma^s\diff \mu \leq C \Abs{\frac{\lambda}{\CalV(f)}}\norm{A}{L^{\infty}([\gamma>0])}\left(\int\abs{\phi}^2\gamma^s\diff\mu+\int_{[\gamma>0]}\abs{A}^2\diff\mu\right).
			\end{align}
			Now for the first $\lambda$-term in \eqref{eq:nabla^mApartial_tgamma}, we use Young's inequality twice to obtain
			\begin{align}
				&\frac{\lambda}{\CalA(f)}\int \abs{\phi}^2\gamma^{s-1}\langle D\tilde{\gamma}\circ f, \nu\rangle  H\diff \mu \\
				&\quad \leq C \frac{\lambda^2}{\CalA(f)^2}\int
				\abs{\phi}^2\gamma^s\diff\mu + C \int\abs{\phi}^2\abs{A}^2\gamma^{s-2}\diff \mu \\
				&\quad \leq C \frac{\lambda^2}{\CalA(f)^2}\int
				\abs{\phi}^2\gamma^s\diff\mu + C\norm{A}{L^{\infty}([\gamma>0]))}^4\int\abs{\phi}^2\gamma^s\diff\mu+ C \int\abs{\phi}^2\gamma^{s-4}\diff \mu.
			\end{align}
			For the second $\lambda$-term in \eqref{eq:nabla^mApartial_tgamma}, we can use Young's inequality with $p=\frac{4}{3}$ and $q=4$ to estimate
			\begin{align}
				 \frac{\lambda}{\CalV(f)}\int \abs{\phi}^2\gamma^{s-1}\langle D\tilde{\gamma}\circ f, \nu\rangle  \diff \mu \leq C\frac{\abs{\lambda}^{\frac{4}{3}}}{\abs{\CalV(f)}^{\frac{4}{3}}} \int\abs{\phi}^2\gamma^s \diff \mu + C \int\abs{\phi}^2\gamma^{s-4}\diff \mu.
			\end{align}
		
			Choosing $\varepsilon>0$ sufficiently small and absorbing, by 
			\Cref{lem:KSGFL3.2},
			 \eqref{eq:nabla^mApartial ohne dtgamma} and \eqref{eq:nabla^mApartial_tgamma}
			\begin{align}
					&\frac{\diff}{\diff t} \int |\phi|^2\gamma^s\diff \mu + \frac{3}{4}\int \abs{\nabla^2 \phi}\gamma^s\diff \mu \\
					&\quad\leq  \int \left(P^{m+2}_3(A)+P^m_5(A)\right)*\phi\gamma^s\diff \mu +
					\int \abs{\phi}^2\gamma^{s-1}\langle D\tilde{\gamma}\circ f, \nu\rangle \left(-{\Delta H} - |A^0|^2 H \right)\diff \mu\\
					&\qquad + C\left(\frac{\lambda^2}{\CalA(f)^2}+ \frac{\abs{\lambda}^{\frac{4}{3}}}{\abs{\CalV(f)}^{\frac{4}{3}}}+\norm{A}{L^{\infty}([\gamma>0])}^4\right) \left( \int\abs{\phi}^2\gamma^s\diff \mu +\int_{[\gamma>0]} \abs{A}^2\diff\mu\right)
					\\
					&\qquad + C\int \abs{\phi}^2\gamma^{s-4}+ \int\abs{\phi}^2\gamma^{s-4}\left(\abs{\nabla \gamma}^4 + \gamma^2\abs{\nabla^2\gamma}^2\right)\diff \mu ,\label{eq:D6}
			\end{align}
			where \FFF we \EEE used Young's inequality to obtain the correct powers of $\lambda$ and $\norm{A}{L^{\infty}([\gamma>0])}$.
			Now, all the terms involving $\lambda$ on the right hand side of \eqref{eq:D6} are as in the statement.  For the second {and the last} term in \eqref{eq:D6}, one may proceed exactly as in the proof of \cite[Proposition 3.3]{KSGF}. This way, one creates additional terms which can be estimated by
			\begin{align}
				&\int\abs{\phi}^2\gamma^{s-4}\diff \mu + \int\abs{\nabla\phi}^2\gamma^{s-2}\diff \mu\leq \varepsilon \int|\nabla^2\phi|^2\gamma^s\diff \mu + C_\varepsilon \int_{[\gamma>0]}\abs{A}^2\gamma^{s-4-2m}\diff \mu,
			\end{align}
			for every $\varepsilon>0$, using twice the interpolation inequality \cite[Corollary 5.3]{KSGF} (which trivially also holds in the case $k=m=0$). The first term on the right hand side of \eqref{eq:D6} can then be estimated by means of \cite[(4.15)]{KSGF}. After choosing $\varepsilon>0$ small enough and absorbing, the claim follows since $s\geq 2m+4$ and $\gamma\leq 1$.
		}
	\end{proof}

\section*{Acknowledgments}
This work was supported by the German Research Foundation (DFG) under Grant 404870139 and the Austrian Science Fund (FWF) under Grant P 32788-N. 
The author would like to thank Anna Dall’Acqua for helpful discussions and comments. \FFF In addition, the author is
grateful to the referees for their careful reading and their valuable comments on the
original manuscript.\EEE

\bibliography{Lib}

\begin{thebibliography}{10}

\bibitem{Aleksandrov}
A.~D. Aleksandrov.
\newblock Uniqueness theorems for surfaces in the large. {V}.
\newblock {\em Amer. Math. Soc. Transl. (2)}, 21:412--416, 1962.

\bibitem{Aubin}
T.~Aubin.
\newblock {\em Nonlinear analysis on manifolds. {M}onge--{A}mp\`ere equations},
  volume 252 of {\em Grundlehren der Mathematischen Wissenschaften}.
\newblock Springer-Verlag, New York, 1982.

\bibitem{Blatt}
S.~Blatt.
\newblock A singular example for the {W}illmore flow.
\newblock {\em Analysis (Munich)}, 29(4):407--430, 2009.

\bibitem{BlattHelfrich}
S.~Blatt.
\newblock A note on singularities in finite time for the {$L^2$} gradient flow
  of the {H}elfrich functional.
\newblock {\em J. Evol. Equ.}, 19(2):463--477, 2019.

\bibitem{Bryant1984}
R.~L. Bryant.
\newblock A duality theorem for {W}illmore surfaces.
\newblock {\em Journal of Differential Geometry}, 20(1):23--53, 1984.

\bibitem{Canham}
P.~Canham.
\newblock The minimum energy of bending as a possible explanation of the
  biconcave shape of the human red blood cell.
\newblock {\em J. Theor. Biol.}, 26(1):61--81, 1970.

\bibitem{CFS09}
R.~Chill, E.~Fa\v{s}angov\'{a}, and R.~Sch\"{a}tzle.
\newblock Willmore blowups are never compact.
\newblock {\em Duke Math. J.}, 147(2):345--376, 2009.

\bibitem{CohnAlgebra}
P.~M. Cohn.
\newblock {\em Basic algebra}.
\newblock Springer-Verlag London, Ltd., London, 2003.
\newblock Groups, rings and fields.

\bibitem{DMSS20}
A.~Dall'Acqua, M.~M\"uller, R.~Sch\"atzle, and A.~Spener.
\newblock {The Willmore flow of tori of revolution}, 2020.
\newblock arXiv:2005.13500. To appear in \textit{Anal. PDE}.

\bibitem{DroskeRumpf}
M.~Droske and M.~Rumpf.
\newblock A level set formulation for {W}illmore flow.
\newblock {\em Interfaces Free Bound.}, 6(3):361--378, 2004.

\bibitem{FJM02}
G.~Friesecke, R.~D. James, and S.~M\"{u}ller.
\newblock A theorem on geometric rigidity and the derivation of nonlinear plate
  theory from three-dimensional elasticity.
\newblock {\em Comm. Pure Appl. Math.}, 55(11):1461--1506, 2002.

\bibitem{Hawking}
S.~W. Hawking.
\newblock Gravitational radiation in an expanding universe.
\newblock {\em J. Math. Phys.}, 9(4):598--604, 1968.

\bibitem{Helfrich}
W.~Helfrich.
\newblock Elastic properties of lipid bilayers: Theory and possible
  experiments.
\newblock {\em Zeitschrift für Naturforschung C}, 28(11):693--703, 1973.

\bibitem{Hopf}
H.~Hopf.
\newblock {\em Differential geometry in the large}, volume 1000 of {\em Lecture
  Notes in Mathematics}.
\newblock Springer-Verlag, Berlin, 1983.

\bibitem{Huisken}
G.~Huisken.
\newblock The volume preserving mean curvature flow.
\newblock {\em J. Reine Angew. Math.}, 382:35--48, 1987.

\bibitem{Jachan}
F.~Jachan.
\newblock {\em Area preserving Willmore flow in asymptotically Schwarzschild
  manifolds}.
\newblock PhD thesis, FU Berlin, 2014.

\bibitem{KellerMondinoRiviere}
L.~G.~A. Keller, A.~Mondino, and T.~Rivi\`ere.
\newblock Embedded surfaces of arbitrary genus minimizing the {W}illmore energy
  under isoperimetric constraint.
\newblock {\em Arch. Ration. Mech. Anal.}, 212(2):645--682, 2014.

\bibitem{Kuhnel}
W.~K\"{u}hnel.
\newblock {\em Differential geometry}, volume~16 of {\em Student Mathematical
  Library}.
\newblock American Mathematical Society, Providence, RI, 2002.
\newblock Curves---surfaces---manifolds, Translated from the 1999 German
  original by Bruce Hunt.

\bibitem{KusnerMcGrath}
R.~Kusner and P.~McGrath.
\newblock On the {C}anham problem: bending energy minimizers for any genus and
  isoperimetric ratio.
\newblock {\em Arch. Ration. Mech. Anal.}, 247(1):Paper No. 10, 14, 2023.

\bibitem{KSSI}
E.~Kuwert and R.~Sch\"{a}tzle.
\newblock The {W}illmore flow with small initial energy.
\newblock {\em J. Differential Geom.}, 57(3):409--441, 2001.

\bibitem{KSGF}
E.~Kuwert and R.~Sch\"{a}tzle.
\newblock Gradient flow for the {W}illmore functional.
\newblock {\em Comm. Anal. Geom.}, 10(2):307--339, 2002.

\bibitem{KSRemovability}
E.~Kuwert and R.~Sch\"{a}tzle.
\newblock Removability of point singularities of {W}illmore surfaces.
\newblock {\em Ann. of Math. (2)}, 160(1):315--357, 2004.

\bibitem{KSLectureNotes}
E.~Kuwert and R.~Sch\"{a}tzle.
\newblock The {W}illmore functional.
\newblock In {\em Topics in modern regularity theory}, volume~13 of {\em CRM
  Series}, pages 1--115. Ed. Norm., Pisa, 2012.

\bibitem{Lee}
J.~M. Lee.
\newblock {\em Introduction to smooth manifolds}, volume 218 of {\em Graduate
  Texts in Mathematics}.
\newblock Springer, New York, second edition, 2013.

\bibitem{Lengeler}
D.~Lengeler.
\newblock Asymptotic stability of local {H}elfrich minimizers.
\newblock {\em Interfaces Free Bound.}, 20(4):533--550, 2018.

\bibitem{LiYau}
P.~Li and S.~T. Yau.
\newblock A new conformal invariant and its applications to the {W}illmore
  conjecture and the first eigenvalue of compact surfaces.
\newblock {\em Invent. Math.}, 69(2):269--291, 1982.

\bibitem{Loja63}
S.~{\L}ojasiewicz.
\newblock Une propri\'{e}t\'{e} topologique des sous-ensembles analytiques
  r\'{e}els.
\newblock In {\em Les \'{E}quations aux {D}\'{e}riv\'{e}es {P}artielles
  ({P}aris)}, pages 87--89. \'{E}ditions du Centre National de la Recherche
  Scientifique, Paris, 1963.

\bibitem{Loja65}
S.~{\L}ojasiewicz.
\newblock Sur les ensembles semi-analytiques.
\newblock {I}.H.E.S., Bures-sur-Yvette, 1965.

\bibitem{MarquesNeves}
F.~C. Marques and A.~Neves.
\newblock Min-max theory and the {W}illmore conjecture.
\newblock {\em Ann. of Math. (2)}, 179(2):683--782, 2014.

\bibitem{McCoyWheelerClassification}
J.~McCoy and G.~Wheeler.
\newblock A classification theorem for {H}elfrich surfaces.
\newblock {\em Math. Ann.}, 357(4):1485--1508, 2013.

\bibitem{McCoyWheeler}
J.~McCoy and G.~Wheeler.
\newblock Finite time singularities for the locally constrained {W}illmore flow
  of surfaces.
\newblock {\em Comm. Anal. Geom.}, 24(4):843--886, 2016.

\bibitem{McCoyWheelerWilliams}
J.~McCoy, G.~Wheeler, and G.~Williams.
\newblock Lifespan theorem for constrained surface diffusion flows.
\newblock {\em Math. Z.}, 269(1-2):147--178, 2011.

\bibitem{MondinoScharrer2}
A.~Mondino and C.~Scharrer.
\newblock A strict inequality for the minimization of the {W}illmore functional
  under isoperimetric constraint.
\newblock {\em Advances in Calculus of Variations}, page 000010151520210002,
  2021.

\bibitem{Rupp}
F.~Rupp.
\newblock On the {{\L}}ojasiewicz--{S}imon gradient inequality on submanifolds.
\newblock {\em J. Funct. Anal.}, 279(8):108708, 2020.

\bibitem{Rupp_2022}
F.~Rupp.
\newblock {\em Constrained gradient flows for Willmore-type functionals}.
\newblock PhD thesis, Universität Ulm, 2022.

\bibitem{RuppVolumePreserving}
F.~Rupp.
\newblock The volume-preserving {W}illmore flow.
\newblock {\em Nonlinear Anal.}, 230:Paper No. 113220, 30, 2023.

\bibitem{Scharrer}
C.~Scharrer.
\newblock Embedded {D}elaunay tori and their {W}illmore energy.
\newblock {\em Nonlinear Anal.}, 223:Paper No. 113010, 23, 2022.

\bibitem{Schygulla}
J.~Schygulla.
\newblock Willmore minimizers with prescribed isoperimetric ratio.
\newblock {\em Arch. Ration. Mech. Anal.}, 203(3):901--941, 2012.

\bibitem{MR0727703}
L.~Simon.
\newblock Asymptotics for a class of nonlinear evolution equations, with
  applications to geometric problems.
\newblock {\em Ann. of Math. (2)}, 118(3):525--571, 1983.

\bibitem{SimonWillmore}
L.~Simon.
\newblock Existence of surfaces minimizing the {W}illmore functional.
\newblock {\em Comm. Anal. Geom.}, 1(2):281--326, 1993.

\bibitem{Simons}
J.~Simons.
\newblock Minimal varieties in {R}iemannian manifolds.
\newblock {\em Ann. of Math. (2)}, 88:62--105, 1968.

\bibitem{MR2898774}
G.~Wheeler.
\newblock Surface diffusion flow near spheres.
\newblock {\em Calc. Var. Partial Differential Equations}, 44(1-2):131--151,
  2012.

\bibitem{MR1261641}
T.~J. Willmore.
\newblock {\em Riemannian geometry}.
\newblock Oxford Science Publications. The Clarendon Press, Oxford University
  Press, New York, 1993.

\bibitem{Zwillinger}
D.~Zwillinger.
\newblock {\em C{RC} standard mathematical tables and formulae}.
\newblock CRC Press, Boca Raton, FL, 32nd edition, 2012.

\end{thebibliography}
\bibliographystyle{abbrv}

\end{document}